\documentclass[aos,preprint]{imsart}

\RequirePackage[utf8]{inputenc}
\RequirePackage{multirow}
\RequirePackage{amsthm,amsmath,amsfonts,amssymb}
\RequirePackage{natbib}
\RequirePackage[colorlinks,citecolor=blue,urlcolor=blue]{hyperref}
\RequirePackage{graphicx}
\usepackage{enumerate}
\usepackage{tabularx,ragged2e,booktabs,caption}
\usepackage[flushleft]{threeparttable}

\arxiv{arXiv:0000.0000}

\startlocaldefs
\numberwithin{equation}{section}
\theoremstyle{plain}
\newtheorem{thm}{Theorem}[section]
\newtheorem{remark}{Remark}[section]

\newtheorem{lemma}{Lemma}[section]

\newtheorem{prop}{Proposition}[section]
\newtheorem{corollary}{Corollary}[section]
\newcommand{\mbf}[1]{\boldsymbol{#1}}
\newcommand{\mbb}[1]{\mathbb{#1}}
\newcommand{\mcal}[1]{\mathcal{#1}}
\newcommand{\mrm}[1]{\textrm{#1}}
\newcommand\numberthis{\addtocounter{equation}{1}\tag{\theequation}}
\endlocaldefs

\begin{document}

\begin{frontmatter}
\title{Asymptotic Properties of Neural Network Sieve Estimators}
\runtitle{Asymptotics for Neural Networks}

\begin{aug}
\author{\fnms{Xiaoxi} \snm{Shen}\thanksref{m1,m2}
\ead[label=e1]{shenxia4@stt.msu.edu}},
\author{\fnms{Chang} \snm{Jiang}\thanksref{m2}
\ead[label=e2]{jiangc@ufl.edu}}
\author{\fnms{Lyudmila} \snm{Sakhanenko}\thanksref{m1}
\ead[label=e3]{luda@stt.msu.edu}}
\and
\author{\fnms{Qing} \snm{Lu}\thanksref{t1, m2}
\ead[label=e4]{lucienq@ufl.edu}}

\thankstext{t1}{To whom correspondence should be addressed.}
\runauthor{X. Shen et al.}

\affiliation{Michigan State University\thanksmark{m1}}
\affiliation{University of Florida\thanksmark{m2}}

\address{Department of Statistics and Probability\\
Michigan State University\\
East Lansing, MI 48824, USA\\
\printead{e1}\\
\phantom{E-mail:\ }}

\address{Department Biostatistics\\
University of Florida\\
Gainesville, FL 32611, USA\\
\printead{e2}}

\address{Department of Statistics and Probability\\
Michigan State University\\
East Lansing, MI 48824, USA\\
\printead{e3}\\
\phantom{E-mail:\ }}

\address{Department Biostatistics\\
University of Florida\\
Gainesville, FL 32611, USA\\
\printead{e4}}
\end{aug}

\begin{abstract}
Neural networks have become one of the most popularly used methods in machine learning and artificial intelligence. Due to the universal approximation theorem \citep{hornik1989multilayer}, a neural network with one hidden layer can approximate any continuous function on compact support as long as the number of hidden units is sufficiently large. Statistically, a neural network can be classified into a nonlinear regression framework. However, if we consider it parametrically, due to the unidentifiability of the parameters, it is difficult to derive its asymptotic properties. Instead, we consider the estimation problem in a nonparametric regression framework and use the results from sieve estimation to establish the consistency, the rates of convergence and the asymptotic normality of the neural network estimators. We also illustrate the validity of the theories via simulations.
\end{abstract}


\begin{keyword}
\kwd{Empirical Processes}
\kwd{Entropy Integral}
\end{keyword}

\end{frontmatter}

\section{Introduction}
With the success of machine learning and artificial intelligence in researches and industry, neural networks have become popularly used methods nowadays. Many newly developed machine learning methods are based on deep neural networks and have achieved great classification and prediction accuracy. We refer interested readers to \citet{goodfellow2016deep} for more background and details. In classical statistical learning theory, the consistency and the rate of convergence of the empirical risk minimization principle are of great interest. Many upper bounds have been established for the empirical risk and the sample complexity based on the growth function and the Vapnik-Chervonenkis dimension (see for example, \citet{vapnik1998statistical, anthony2009neural, devroye2013probabilistic}). However, few studies have focused on the asymptotic properties for neural networks. As Thomas J. Sargent said, ``artificial intelligence is actually statistics, but in a very gorgeous phrase, it is statistics." So it is natural and worthwhile to explore whether neural networks possess nice asymptotic properties. As if they do, it may be possible to conduct statistical inference based on neural networks. Throughout this paper, we will focus on the asymptotic properties of neural networks with one hidden layer.

In statistics, fitting a neural network with one hidden layer can be viewed as a parametric nonlinear regression problem:
$$
y_i=\alpha_0+\sum_{j=1}^{r}\alpha_j\sigma\left(\mbf{\gamma}_j^T\mbf{x}_i+\gamma_{0,j}\right)+\epsilon_i,
$$
where $\epsilon_1,\ldots,\epsilon_n$ are i.i.d. random errors with $\mbb{E}[\epsilon]=0$ and $\mbb{E}[\epsilon^2]=\sigma^2<\infty$ and $\sigma(\cdot)$ is an activation function such as $\sigma(z)=1/(1+e^{-z})$, which is used in this paper. \citet{white2001statistical} obtained the asymptotic distribution of the resulting estimators under the assumption of the true parameters being unique. In fact, the authors implicitly assumed that the number of hidden units $r$ is known. However, even if we assume that we know the number of hidden units, it is difficult to establish the asymptotic properties for the parameter estimators. In section \ref{Sec: parameter inconsistency}, we conducted a simulation based on a single-layer neural network with 2 hidden units. Even for such a simple model, the simulation result suggests that it is unlikely to obtain consistent estimators. Moreover, since the number of hidden units is usually unknown in practice, such an assumption can be easily violated. For example, as pointed out in \citet{fukumizu1996regularity} and \citet{fukumizu2003likelihood}, if the true function is $f_0(x)=\alpha\sigma(\gamma x)$, (i.e., the true number of hidden units is 1), and we fit the model using a neural network with two hidden units, then any parameter $\mbf{\theta}=[\alpha_0,\alpha_1,\ldots,\alpha_r,\gamma_{0,1},\ldots,\gamma_{0,r},\mbf{\gamma}_1^T,\ldots,\mbf{\gamma}_r^T]^T$ in the high-dimensional set
\begin{align*}
\{\mbf{\theta}:\gamma_1=\gamma,\alpha_1=\alpha,\gamma_{0,1}=\gamma_{0,2}=\alpha_2=\alpha_0=0\}\cup\\
\{\mbf{\theta}:\gamma_1=\gamma_2=\gamma,\gamma_{0,1}=\gamma_{0,2}=\alpha_0=0,\alpha_1+\alpha_2=\alpha\}
\end{align*}
realizes the true function $f_0(x)$. Therefore, when the number of hidden units is unknown, the parameters in this parametric nonlinear regression problem are unidentifiable. Theorem 1 in \citet{wu1981asymptotic} showed that a necessary condition for the weak consistency of nonlinear least square estimators is
$$
\sum_{i=1}^n[f(\mbf{x}_i,\mbf{\theta})-f(\mbf{x}_i,\mbf{\theta}')]^2\to\infty,\mrm{ as }n\to\infty,
$$
for all $\mbf{\theta}\neq\mbf{\theta}'$ in the parameter space as long as the error distribution has finite Fisher information. Such a condition implies that when the parameters are not identifiable, the resulting nonlinear least squares estimators will be inconsistent, which hinders further explorations on the asymptotic properties for the neural network estimators. \citet{liu2003asymptotics} and \citet{zhu2006asymptotics} proposed techniques to conduct hypothesis testing under loss of identifiability. However, their theoretical results are not easy to implement in the neural network setting.

Even though a function can have different neural network parametrizations, the function itself can be considered as unique. Moreover, due to the Universal Approximation Theorem \citep{hornik1989multilayer}, any continuous function on compact support can be approximated arbitrarily well by a neural network with one hidden layer. So it seems natural to consider it as a nonparametric regression problem and approximate the underlying function class through a class of neural networks with one hidden layer. Specifically, suppose that the true nonparametric regression model is
$$
y_i=f_0(\mbf{x}_i)+\epsilon_i,
$$
where $\epsilon_1,\ldots,\epsilon_n$ are i.i.d. random variables defined on a complete probability space $(\Omega,\mcal{A},\mbb{P})$ with $\mbb{E}[\epsilon]=0$, $\mrm{Var}[\epsilon]=\sigma^2<\infty$; $\mbf{x}_1,\ldots,\mbf{x}_n\in\mcal{X}\subset\mbb{R}^d$ are vectors of covariates with $\mcal{X}$ being a compact set in $\mbb{R}^d$ and $f_0$ is an unknown function needed to be estimated. We assume that $f_0\in\mcal{F}$, where $\mcal{F}$ is the class of continuous functions with compact supports. Clearly, $f_0$ minimizes the population criterion function
\begin{align*}
Q_n(f) & =\mbb{E}\left[\frac{1}{n}\sum_{i=1}^n(y_i-f(\mbf{x}_i))^2\right]\\
	& =\frac{1}{n}\sum_{i=1}^n(f(\mbf{x}_i)-f_0(\mbf{x}_i))^2+\sigma^2.
\end{align*}
A least squares estimator of the regression function can be obtained by minimizing the empirical squared error loss $\mbb{Q}_n(f)$:
$$
\hat{f}_n=\mrm{argmin}_{f\in\mcal{F}}\mbb{Q}_n(f)=\mrm{argmin}_{f\in\mcal{F}}\frac{1}{n}\sum_{i=1}^n(y_i-f(\mbf{x}_i))^2.
$$
However, if the class of functions $\mcal{F}$ is too rich, the resulting least squares estimator may have undesired properties, such as inconsistency \citep{van2000empirical, shen1994convergence, shen1997methods}. Instead, we can optimize the squared error loss over some less complex function space $\mcal{F}_n$, which is an approximation of $\mcal{F}$ while the approximation error tends to 0 as the sample size increases. In the language of \citet{grenanderabstract}, such a sequence of function classes is known as a sieve. More precisely, we consider a sequence of function classes,
$$
\mcal{F}_1\subseteq\mcal{F}_2\subseteq\cdots\subseteq\mcal{F}_n\subseteq\mcal{F}_{n+1}\subseteq\cdots\subseteq\mcal{F},
$$
approximating $\mcal{F}$ in the sense that $\bigcup_{n=1}^\infty\mcal{F}_n$ is dense in $\mcal{F}$. In other words, for each $f\in\mcal{F}$, there exists $\pi_nf\in\mcal{F}_n$ such that $d(f,\pi_nf)\to0$ as $n\to\infty$, where $d(\cdot,\cdot)$ is some pseudo-metric defined on $\mcal{F}$. With some abuse of notation, an approximate sieve estimator $\hat{f}_n$ is defined to be
\begin{equation}\label{Eq: SieveEstimator}
\mbb{Q}_n(\hat{f}_n)\leq\inf_{f\in\mcal{F}_n}\mbb{Q}_n(f)+\mcal{O}_p(\eta_n),
\end{equation}
where $\eta_n\to0$ as $n\to\infty$.

Throughout the rest of the paper, we focus on the sieve of neural networks with one hidden layer and sigmoid activation function. Specifically, we let
\begin{align*}
\mcal{F}_{r_n} & =\left\{\alpha_0+\sum_{j=1}^{r_n}\alpha_j\sigma\left(\mbf{\gamma}_j^T\mbf{x}+\gamma_{0,j}\right):\mbf{\gamma}_j\in\mbb{R}^d, \alpha_j, \gamma_{0,j}\in\mbb{R},\right.\\
	& \qquad\sum_{j=0}^{r_n}|\alpha_j|\leq V_n\mrm{ for some }V_n>4\left.\mrm{ and }\max_{1\leq j\leq r_n}\sum_{i=0}^d|\gamma_{i,j}|\leq M_n\mrm{ for some }M_n>0\right\},\numberthis\label{Eq: NNSieve}
\end{align*}
where $r_n, V_n, M_n\uparrow\infty$ as $n\to\infty$. Such method has been discussed in previous literatures (e.g. \citet{white1989learning} and \citet{white1990connectionist}). In those papers, consistency of the neural network sieve estimators has been established under random designs. However, there are few results on the asymptotic distribution of the neural network sieve estimators, which will be established in this paper. Throughout this paper, we focus on the fixed design. \citet{hornik1989multilayer} showed that $\bigcup_n\mcal{F}_{r_n}$ is dense in $\mcal{F}$ under the sup-norm. But when considering the asymptotic properties of the sieve estimators, we use the pseudo-norm $\|f\|_n^2=n^{-1}\sum_{i=1}^nf^2(\mbf{x}_i)$ (see Proposition \ref{Prop: Pseudo-metric Parameter space} in the Appendix) defined on $\mcal{F}$ and $\mcal{F}_{r_n}$.

In section 2, we discuss the existence of neural network sieve estimators. The weak consistency and rate of convergence of the neural network sieve estimators will be established in section 3 and section 4, respectively. Section 5 focuses on the asymptotic distribution of the neural network sieve estimators. Simulation results are presented in section 6.

\textit{Notations}: Throughout the rest of the paper, bold font alphabetic letters and Greek letters are vectors. $C(\mcal{X})$ is the set of continuous functions defined on $\mcal{X}$. The symbol $\lesssim$ means ``bounded above up to a universal constant" and $a_n\sim b_n$ means $\frac{a_n}{b_n}\to1$ as $n\to\infty$. For a pseudo-metric space $(T,d)$, $N(\epsilon,T,d)$ is its covering number, which is the minimum number of $\epsilon$-balls needed to cover $T$. Its natural logarithm is the entropy number and is denoted by $H(\epsilon,T,d)$.

\section{Existence}
A natural question to ask is whether the sieve estimator based on neural networks exists. Before addressing this question, we first study some properties of $\mcal{F}_{r_n}$. Proposition \ref{Prop: Sigmoid-Lipschitz} shows that the sigmoid function is a Lipschitz function with Lipschitz constant $L=1/4$.

\begin{prop}\label{Prop: Sigmoid-Lipschitz}
A sigmoid function $\sigma(z)=e^z/(1+e^z)$ is a Lipschitz function on $\mbb{R}$ with Lipschitz constant $1/4$.
\end{prop}

\begin{proof}
For all $z_1,z_2\in\mbb{R}$, $\sigma(z)$ is continuous on $[z_1,z_2]$ and is differentiable on $(z_1,z_2)$. Note that
\begin{align*}
\sigma'(z) & =\sigma(z)(1-\sigma(z))\leq\frac{1}{4}\quad\forall z\in\mbb{R}.
\end{align*}
By using the Mean Value Theorem, we know that
$$
\sigma(z_1)-\sigma(z_2)=\sigma'(\lambda z_1+(1-\lambda)z_2)(z_1-z_2),
$$
for some $\lambda\in[0,1]$. Hence
$$
|\sigma(z_1)-\sigma(z_2)|=|\sigma'(\lambda z_1+(1-\lambda)z_2)||z_1-z_2|\leq\frac{1}{4}|z_1-z_2|,
$$
which means that $\sigma(z)$ is a Lipschitz function on $\mbb{R}$ with Lipschitz constant $1/4$.
\end{proof}

The second proposition provides an upper bound for the envelope function $\sup_{f\in\mcal{F}_{r_n}}|f|$. 
\begin{prop}\label{Prop: Envelope}
For each fixed $n$,
$$
\sup_{f\in\mcal{F}_{r_n}}\|f\|_\infty\leq V_n.
$$
\end{prop}

\begin{proof}
For any $f\in\mcal{F}_{r_n}$ with a fixed $n$ and all $\mbf{x}\in\mcal{X}$, we have
\begin{align*}
|f(\mbf{x})| & =\left|\alpha_0+\sum_{j=1}^{r_n}\alpha_j\sigma\left(\mbf{\gamma}_j^T\mbf{x}+\gamma_{0,j}\right)\right|\\
	& \leq|\alpha_0|+\sum_{j=1}^{r_n}|\alpha_j|\sigma\left(\mbf{\gamma}_j^T\mbf{x}+\gamma_{0,j}\right)\leq\sum_{j=0}^{r_n}|\alpha_j|\leq V_n.
\end{align*}
Since the right hand side does not depend on $\mbf{x}$ and $f$, we get
$$
\sup_{f\in\mcal{F}_{r_n}}\|f\|_\infty=\sup_{f\in\mcal{F}_{r_n}}\sup_{\mbf{x}\in\mcal{X}}|f(\mbf{x})|\leq V_n.
$$
\end{proof}

Now we quote a general result from \citet{white1991some}. The theorem tells us that under some mild conditions, there exists a sieve approximate estimator and such an estimator is also measurable.

\begin{thm}[Theorem 2.2 in \citet{white1991some}]\label{Thm: ExistenceOfSieve}
Let $(\Omega,\mcal{A},\mbb{P})$ be a complete probability space and let $(\Theta,\rho)$ be a pseudo-metric space. Let $\{\Theta_n\}$ be a sequence of compact subsets of $\Theta$. Let $\mbb{Q}_n:\Omega\times\Theta_n\to\bar{\mbb{R}}$ be $\mcal{A}\otimes\mcal{B}(\Theta_n)/\mcal{B}(\bar{\mbb{R}})$-measurable, and suppose that for each $\omega\in\Omega$, $\mbb{Q}_n(\omega,\cdot)$ is lower semicontinuous on $\Theta_n$, $n=1,2,\ldots$. Then for each $n=1,2,\ldots$, there exists $\hat{\theta}_n:\Omega\to\Theta_n$, $\mcal{A}/\mcal{B}(\Theta_n)$-measurable such that for each $\omega\in\Omega$, $\mbb{Q}_n(\omega,\hat{\theta}_n(\omega))=\inf_{\theta\in\Theta_n}\mbb{Q}_n(\omega,\theta)$.
\end{thm}

Note that 
\begin{align*}
\mbb{Q}_n(f) & =\frac{1}{n}\sum_{i=1}^n(y_i-f(\mbf{x}_i))^2\\
	& =\frac{1}{n}\sum_{i=1}^n(f_0(\mbf{x}_i)+\epsilon_i-f(\mbf{x}_i))^2\\
	& =\frac{1}{n}\sum_{i=1}^n(f(\mbf{x}_i)-f_0(\mbf{x}_i))^2-2\frac{1}{n}\sum_{i=1}^n\epsilon_i(f(\mbf{x}_i)-f_0(\mbf{x}_i))+\frac{1}{n}\sum_{i=1}^n\epsilon_i^2.
\end{align*}
Since the randomness only comes from $\epsilon_i$'s, it is clear that $\mbb{Q}_n$ is a measurable function and for a fixed $\omega$, $\mbb{Q}_n$ is continuous in $f$. Therefore, to show the existence of the sieve estimator, it suffices to show that $\mcal{F}_{r_n}$ is compact in $C(\mcal{X})$, which is proved in the following lemma.

\begin{lemma}\label{Lm: Compactness of NN Sieve Class}
Let $\mcal{X}$ be a compact subset of $\mbb{R}^d$. Then for each fixed $n$, $\mcal{F}_{r_n}$ is a compact set.
\end{lemma}

\begin{proof}
For each fixed $n$, let $\mbf{\theta}_n=[\alpha_0,\ldots,\alpha_{r_n},\mbf{\gamma}_{0,1},\ldots,\gamma_{0,r_n},\mbf{\gamma}_1^T,\ldots,\mbf{\gamma}_{r_n}^T]^T$ belong to $[-V_n,V_n]^{r_n+1}\times[-M_n,M_n]^{r_n(d+1)}:=\Theta_n$. If $n$ is fixed, $\Theta_n$ is a bounded closed set and hence it is a compact set in $\mbb{R}^{r_n(d+2)+1}$. Consider a map
\begin{align*}
H:(\Theta_n,\|\cdot\|_2) & \to(\mcal{F}_{r_n},\|\cdot\|_n)\\
	\mbf{\theta}_n & \mapsto H(\mbf{\theta}_n)=\alpha_0+\sum_{j=1}^{r_n}\alpha_j\sigma\left(\mbf{\gamma}_j^T\mbf{x}+\gamma_{0,j}\right)
\end{align*}
Note that $\mcal{F}_{r_n}=H(\Theta_n)$. Therefore, to show that $\mcal{F}_{r_n}$ is a compact set, it suffices to show that $H$ is a continuous map due to the compactness of $\Theta_n$. Let $\mbf{\theta}_{1,n},\mbf{\theta}_{2,n}\in\Theta_n$, then
\begin{align*}
	& \|H(\mbf{\theta}_{1,n})-H(\mbf{\theta}_{2,n})\|_n^2\\
   = & \frac{1}{n}\sum_{i=1}^n\left[\alpha_0^{(1)}+\sum_{j=1}^{r_n}\alpha_j^{(1)}\sigma\left(\mbf{\gamma}_j^{(1)^T}\mbf{x}_i+\gamma_{0,j}^{(1)}\right)-\alpha_0^{(2)}-\sum_{j=1}^{r_n}\alpha_j^{(2)}\sigma\left(\mbf{\gamma}_j^{(2)^T}\mbf{x}_i+\gamma_{0,j}^{(2)}\right)\right]^2\\
	\leq & \frac{1}{n}\sum_{i=1}^n\left[\left|\alpha_0^{(1)}-\alpha_0^{(2)}\right|+\sum_{j=1}^{r_n}\left|\alpha_j^{(1)}\sigma\left(\mbf{\gamma}_j^{(1)^T}\mbf{x}_i+\gamma_{0,j}^{(1)}\right)-\alpha_j^{(2)}\sigma\left(\mbf{\gamma}_j^{(2)^T}\mbf{x}_i+\gamma_{0,j}^{(2)}\right)\right|\right]^2\\
	= & \frac{1}{n}\sum_{i=1}^n\left[\left|\alpha_0^{(1)}-\alpha_0^{(2)}\right|+\sum_{j=1}^{r_n}|\alpha_j^{(1)}|\left|\sigma\left(\mbf{\gamma}_j^{(1)^T}\mbf{x}_i+\gamma_{0,j}^{(1)}\right)-\sigma\left(\mbf{\gamma}_j^{(2)^T}\mbf{x}_i+\gamma_{0,j}^{(2)}\right)\right|+\right.\\
	& \hspace{8cm}\left.|\alpha_j^{(1)}-\alpha_j^{(2)}|\sigma\left(\mbf{\gamma}_j^{(2)^T}\mbf{x}_i+\gamma_{0,j}^{(2)}\right)\right]^2\\
	\leq & \frac{1}{n}\sum_{i=1}^n\left[\sum_{j=0}^{r_n}|\alpha_j^{(1)}-\alpha_j^{(2)}|+\frac{V_n}{4}\sum_{j=1}^{r_n}\left|\left(\mbf{\gamma}_j^{(1)}-\mbf{\gamma}_j^{(2)}\right)^T\mbf{x}_i\right|+\left|\gamma_{0,j}^{(1)}-\gamma_{0,j}^{(2)}\right|\right]^2\\
	\leq & \left[\sum_{j=0}^{r_n}|\alpha_j^{(1)}-\alpha_j^{(2)}|+\frac{V_n}{4}(1\vee\|\mbf{x}\|_\infty)\sum_{j=1}^{r_n}\left\|\mbf{\gamma}_j^{(1)}-\mbf{\gamma}_j^{(2)}\right\|_1+\left|\gamma_{0,j}^{(1)}-\gamma_{0,j}^{(2)}\right|\right]^2\\
	\leq & \left(\frac{V_n}{4}(1\vee\|\mbf{x}\|_\infty)\right)^2[r_n(d+1)]\|\mbf{\theta}_{1,n}-\mbf{\theta}_{2,n}\|_2^2.
\end{align*}
Hence, for any $\epsilon>0$, we choose $\delta=\epsilon/\left(\frac{V_n}{4}(1\vee\|\mbf{x}\|_\infty)\sqrt{r_n(d+1)}\right)$,. When $\|\mbf{\theta}_{1,n}-\mbf{\theta}_{2,n}\|_2<\delta$, we have
$$
\|H(\mbf{\theta}_{1,n})-H(\mbf{\theta}_{2,n})\|_n<\epsilon,
$$
which implies that $H$ is a continuous map and hence $\mcal{F}_{r_n}$ is a compact set for each fixed $n$.
\end{proof}

As a corollary of Lemma \ref{Lm: Compactness of NN Sieve Class} and Theorem \ref{Thm: ExistenceOfSieve}, we can easily obtain the existence of sieve estimator.
\begin{corollary}
Based on the notations above, for each $n=1,2,\ldots$, there exists $\hat{f}_n:\Omega\to\mcal{F}_{r_n}$, $\mcal{A}/\mcal{B}(\mcal{F}_{r_n})$-measurable such that $\mbb{Q}_n(\hat{f}_n(\omega))=\inf_{f\in\mcal{F}_{r_n}}\mbb{Q}_n(f)$.
\end{corollary}

\section{Consistency}\label{Sec: Consistency}
In this section, we are going to show the consistency of the neural network sieve estimator. The consistency result leans heavily on the following Uniform Law of Large Numbers. We start by considering a simple case with $V_n\equiv V$ for all $n$. In such a case, $\bigcup_n\mcal{F}_{r_n}$ is not dense in $\mcal{F}$ but rather in a subset of $\mcal{F}$ with functions satisfying a certain smoothness condition.

\begin{lemma}\label{Lm: NNSieveULLN}
Let $\epsilon_1,\ldots,\epsilon_n$ be i.i.d. sub-Gaussian random variables with sub-Gaussian parameter $\sigma_0$. If $[r_n(d+2)+1]\log[r_n(d+2)+1]=o(n)$, we have
$$
\sup_{f\in\mcal{F}_{r_n}}\left|\mbb{Q}_n(f)-Q_n(f)\right|\xrightarrow{p^*}0.
$$
\end{lemma}

\begin{proof}
For any $\delta>0$, we have
\begin{align*}
 & \mbb{P}^*\left(\sup_{f\in\mcal{F}_{r_n}}|\mbb{Q}_n(f)-Q_n(f)|>\delta\right)\\
= & \mbb{P}^*\left(\sup_{f\in\mcal{F}_{r_n}}\left|\frac{1}{n}\sum_{i=1}^n\epsilon_i^2-\sigma^2-2\frac{1}{n}\sum_{i=1}^n\epsilon_i\left(f(\mbf{x}_i)-f_0(\mbf{x}_i)\right)\right|>\delta\right)\\
	\leq & \mbb{P}\left(\left|\frac{1}{n}\sum_{i=1}^n\epsilon_i^2-\sigma^2\right|>\frac{\delta}{2}\right)+\mbb{P}^*\left(\sup_{f\in\mcal{F}_{r_n}}\left|\frac{1}{n}\sum_{i=1}^n\epsilon_i(f(\mbf{x}_i)-f_0(\mbf{x}_i))\right|>\frac{\delta}{4}\right)\\
	:= & (I)+(II).
\end{align*}
For (I), based on the Weak Law of Large Numbers, we know that there exists $N_1>0$ such that for all $n\geq N_1$ we have
$$
(I)=\mbb{P}\left(\left|\frac{1}{n}\sum_{i=1}^n\epsilon_i^2-\sigma^2\right|>\frac{\delta}{2}\right)<\frac{\delta}{2}.
$$
Now, we are going to evaluate (II). From the sub-Gaussianity of $\epsilon_1,\ldots,\epsilon_n$, we know that $\epsilon_i(f(\mbf{x}_i)-f_0(\mbf{x}_i))$ is also sub-Gaussian with mean 0 and sub-Gaussian parameter $\sigma_0|f(\mbf{x}_i)-f_0(\mbf{x}_i)|$. Hence, by using the Hoeffding inequality,
\begin{align*}
\mbb{P}\left(\left|\frac{1}{n}\sum_{i=1}^n\epsilon_i(f(\mbf{x}_i)-f_0(\mbf{x}_i))\right|>\frac{\delta}{4}\right) & =\mbb{P}\left(\left|\sum_{i=1}^n\epsilon_i(f(\mbf{x}_i)-f_0(\mbf{x}_i))\right|>\frac{n\delta}{4}\right)\\
	& \leq2\exp\left\{-\frac{n^2\delta^2}{32\sigma_0^2\sum_{i=1}^n(f(\mbf{x}_i)-f_0(\mbf{x}_i))^2}\right\}.
\end{align*}
From Proposition \ref{Prop: Envelope}, we know that $\sup_{f\in\mcal{F}_{r_n}}\|f\|_n\leq V$. Hence, based on Corollary 8.3 in \citet{van2000empirical}, (II) will have an exponential bound if there exists some constant $C$, $\delta>0$ and $\sigma>0$ satisfying $V>\delta/\sigma$ and
\begin{equation}\label{Eq: DudleyEntropyIntegralBound}
\sqrt{n}\delta\geq2C\left(\int_{\delta/(8\sigma)}^{V}H^{1/2}(u,\mcal{F}_{r_n},\|\cdot\|_n)\mrm{d}u\vee V\right).
\end{equation}
Now, we are going to show that (\ref{Eq: DudleyEntropyIntegralBound}) holds in our case. It follows from Theorem 14.5 in \citet{anthony2009neural}, which gives an upper bound of the covering number for $\mcal{F}_{r_n}$,
$$
N(\epsilon,\mcal{F}_{r_n},\|\cdot\|_\infty)\leq\left(\frac{4e[r_n(d+2)+1]\left(\frac{1}{4}V\right)^2}{\epsilon\left(\frac{1}{4}V-1\right)}\right)^{r_n(d+2)+1}:=\tilde{A}_{r_n,d,V}\epsilon^{-[r_n(d+2)+1]},
$$
where $\tilde{A}_{r_n,d,V}=\left(e[r_n(d+2)+1]V^2/(V-4)\right)^{r_n(d+2)+1}$. By letting 
\begin{align*}
A_{r_n,d,V} & =\log\tilde{A}_{r_n,d,V}-[r_n(d+2)+1]\\
	& =[r_n(d+2)+1]\left(\log\frac{e[r_n(d+2)+1]V^2}{V-4}-1\right)\\
	& =[r_n(d+2)+1]\log\frac{[r_n(d+2)+1]V^2}{V-4},
\end{align*}
and noting that $V^2-eV+4e\geq0$ for all $V$, we have $\log\frac{[r_n(d+2)+1]V^2}{V-4}\geq\log\frac{V^2}{V-4}\geq\log\frac{e(V-4)}{V-4}=1$. Then,
\begin{align*}
H(\epsilon,\mcal{F}_{r_n},\|\cdot\|_\infty) & =\log N(\epsilon,\mcal{F}_{r_n},\|\cdot\|_\infty)\\
	& =\log\tilde{A}_{r_n,d,V}+[r_n(d+2)+1]\log\frac{1}{\epsilon}\\
	& \leq A_{r_n,d,V}+[r_n(d+2)+1]\frac{1}{\epsilon}\quad(\mrm{since }\log x\leq x-1\mrm{ for all }x>0)\\
	& \leq A_{r_n,d,V}\left(1+\frac{1}{\epsilon}\right).
\end{align*}
Note that
$$
\|f\|_n^2=\frac{1}{n}\sum_{i=1}^nf^2(\mbf{x}_i)\leq\left(\sup_{\mbf{x}}|f(\mbf{x})|\right)^2=\|f\|_\infty^2,
$$
we have $H(\epsilon,\mcal{F}_{r_n},\|\cdot\|_n)\leq H(\epsilon,\mcal{F}_{r_n},\|\cdot\|_\infty)$. Then
\begin{align*}
\int_{\delta/(8\sigma)}^{V}H^{1/2}(\epsilon,\mcal{F}_{r_n},\|\cdot\|_n)\mrm{d}\epsilon & \leq A_{r_n,d,V}^{1/2}\int_0^{V}\left(1+\frac{1}{\epsilon}\right)^{1/2}\mrm{d}\epsilon\\
	& = A_{r_n,d,V}^{1/2}\left[\int_0^1\left(1+\frac{1}{\epsilon}\right)^{1/2}\mrm{d}\epsilon+\int_1^{V}\left(1+\frac{1}{\epsilon}\right)^{1/2}\mrm{d}\epsilon\right]\\
	& \leq A_{r_n,d,V}^{1/2}\left[\sqrt{2}\int_0^1\epsilon^{-\frac{1}{2}}\mrm{d}\epsilon+\sqrt{2}(V-1)\right]\\
	& \leq A_{r_n,d,V}^{1/2}\left[2\sqrt{2}+2\sqrt{2}(V-1)\right]\\
	& =2\sqrt{2}A_{r_n,d,V}^{1/2}V.
\end{align*}
Clearly, $2\sqrt{2}A_{r_n,d,V}^{1/2}V\geq V$. Under the assumption of $[r_n(d+2)+1]\log[r_n(d+2)+1]=o(n)$, for any $\delta>0$, there exists $N_2>0$ such that for all $n\geq N_2$,
$$
4\sqrt{2}V\left(\frac{1}{n}A_{r_n,d,V}\right)^{1/2}<\frac{\delta}{4},
$$
i.e. (\ref{Eq: DudleyEntropyIntegralBound}) holds with $C=1$ and $n\geq N_2$. Hence, based on Corollary 8.3 in \citet{van2000empirical}, for $n\geq N_2$,
\begin{equation}\label{Eq: van2000Corollary8.3}
\mbb{P}^*\left(\sup_{f\in\mcal{F}_{r_n}}\left|\frac{1}{n}\sum_{i=1}^n\epsilon_i(f(\mbf{x}_i)-f_0(\mbf{x}_i))\right|>\frac{\delta}{4}\wedge\frac{1}{n}\sum_{i=1}^n\epsilon_i^2\leq\sigma^2\right)\leq\exp\left\{-\frac{n\delta^2}{64V^2}\right\}.
\end{equation}
Since $\int_0^{V}H^{1/2}(\epsilon,\mcal{F}_{r_n},\|\cdot\|_n)\mrm{d}\epsilon<\infty$, we can take $\sigma\to\infty$ in (\ref{Eq: van2000Corollary8.3}) to get
$$
\mbb{P}^*\left(\sup_{f\in\mcal{F}_{r_n}}\left|\frac{1}{n}\sum_{i=1}^n\epsilon_i(f(\mbf{x}_i)-f_0(\mbf{x}_i))\right|>\frac{\delta}{4}\right)\leq\exp\left\{-\frac{n\delta^2}{64V^2}\right\}.
$$
Let $N_3=\frac{64V^2}{\delta^2}\log\frac{2}{\delta}$, then for $n\geq\max\{N_2,N_3\}$, we have
$$
(II)=\mbb{P}^*\left(\sup_{f\in\mcal{F}_{r_n}}\left|\frac{1}{n}\sum_{i=1}^n\epsilon_i(f(\mbf{x}_i)-f_0(\mbf{x}_i))\right|>\frac{\delta}{4}\right)\leq\frac{\delta}{2}.
$$
Thus, we conclude that for any $\delta>0$, by taking $n\geq\max\{N_1,N_2,N_3\}$, we have
$$
\mbb{P}^*\left(\sup_{f\in\mcal{F}_{r_n}}|\mbb{Q}_n(f)-Q_n(f)|>\delta\right)<\delta,
$$
which proves the desired result.
\end{proof}

\begin{remark}
Lemma \ref{Lm: NNSieveULLN} shows that if we have a fixed number of features, the desired Uniform Law of Large Numbers holds when the number of hidden units in the neural network sieve does not grow too fast. 
\end{remark}

Now, we are going to extend the result to a more general case. In Lemma \ref{Lm: NNSieveULLN}, we assume that the errors $\epsilon_1,\ldots,\epsilon_n$ are i.i.d. sub-Gaussian and $V_n\equiv V$. In the following lemma, we are going to relax both restrictions.

\begin{lemma}\label{Lm: NNSieveULLN-2}
Under the assumption of
$$
[r_n(d+2)+1]V_n^2\log(V_n[r_n(d+2)+1]=o(n),\mrm{ as }n\to\infty,
$$
we have
$$
\sup_{f\in\mcal{F}_{r_n}}|\mbb{Q}_n(f)-Q_n(f)|\xrightarrow{p^*}0,\mrm{ as } n\to\infty.
$$
\end{lemma}

\begin{proof}
As in the proof of Lemma \ref{Lm: NNSieveULLN}, it suffices to show that 
\begin{equation}\label{Eq: weightedSumTail}
\mbb{P}^*\left(\sup_{f\in\mcal{F}_{r_n}}\left|\frac{1}{n}\sum_{i=1}^n\epsilon_i(f(\mbf{x}_i)-f_0(\mbf{x}_i))\right|>\frac{\delta}{4}\right)\to0,\mrm{ as }n\to\infty.
\end{equation}
By using the Markov's inequality, (\ref{Eq: weightedSumTail}) holds if we can show
$$
\mbb{E}^*\left[\sup_{f\in\mcal{F}_{r_n}}\left|\frac{1}{n}\sum_{i=1}^n\epsilon_i(f(\mbf{x}_i)-f_0(\mbf{x}_i))\right|\right]\to0,\mrm{ as }n\to\infty.
$$
Note that $\mbb{E}[\epsilon]=0$ and each $f\in\mcal{F}_{r_n}$ has its corresponding parametrization $\mbf{\theta}_n$. Since $\mbf{\theta}_n$ is in a compact set, there exists a sequence $\mbf{\theta}_{n,k}\to\mbf{\theta}_n$ as $k\to\infty$ with $\mbf{\theta}_{n,k}\in\mbb{Q}^{r_n(d+2)+1}\cap([-V_n,V_n]^{r_n+1}\times[-M_n,M_n]^{r_n(d+1)})$. Each $\mbf{\theta}_{n,k}$ corresponds to a function $f_k\in\mcal{F}_{r_n}$. Based on continuity, we have $f_k(\mbf{x})\to f(\mbf{x})$ for each $\mbf{x}\in\mcal{X}$. From Example 2.3.4 in \citet{van1996weak}, we know that $\mcal{F}_{r_n}$ is $P$-measurable. Based on symmetrization inequality, we have
\begin{align*}
 & \mbb{E}^*\left[\sup_{f\in\mcal{F}_{r_n}}\left|\frac{1}{n}\sum_{i=1}^n\epsilon_i(f(\mbf{x}_i)-f_0(\mbf{x}_i))\right|\right]\\
 \leq & 2\mbb{E}_\epsilon\mbb{E}_\xi\left[\sup_{f\in\mcal{F}_{r_n}}\left|\frac{1}{n}\sum_{i=1}^n\xi_i\epsilon_i\left(f(\mbf{x}_i)-f_0(\mbf{x}_i)\right)\right|\right],
\end{align*}
where $\xi_1,\ldots,\xi_n$ are i.i.d. Rademacher random variables independent of $\epsilon_1,\ldots,\epsilon_n$. Based on the Strong Law of Large Numbers, there exists $N_1>0$, such that for all $n\geq N_1$,
$$
\frac{1}{n}\sum_{i=1}^n\epsilon_i^2<\sigma^2+1,\mrm{ a.s.}
$$
For fixed $\epsilon_1,\ldots,\epsilon_n$, $\sum_{i=1}^n\xi_i\epsilon_i(f(\mbf{x}_i)-f_0(\mbf{x}_i))$ is a sub-Gaussian process indexed by $f\in\mcal{F}_{r_n}$. Suppose that $(\Xi, \mcal{C}, \mu)$ is the probability space on which $\xi_1,\ldots,\xi_n$ are defined and let $Y(f,\omega)=\sum_{i=1}^n\xi_i(\omega)\epsilon_i(f(\mbf{x}_i)-f_0(\mbf{x}_i))$ with $f\in\mcal{F}_{r_n}$ and $\omega\in\Xi$. As we have shown above, we have $f_k\to f$ and by continuity, $Y(f_k,\omega)\to Y(f,\omega)$ for any $\omega\in\Xi$. This shows that $\{Y(f,\omega),f\in\mcal{F}_{r_n}\}$ is a separable sub-Gaussian process. Hence Corollary 2.2.8 in \citet{van1996weak} implies that there exists a universal constant $K$ and for any $f_n^*\in\mcal{F}_{r_n}$ with $n\geq N_1$,
\begin{align*}
\mbb{E}_\xi & \left[\sup_{f\in\mcal{F}_{r_n}}\left|\frac{1}{n}\sum_{i=1}^n\xi_i\epsilon_i(f(\mbf{x}_i))-f_0(\mbf{x}_i))\right|\right]\\
	& =\mbb{E}_\xi\left[\frac{1}{\sqrt{n}}\sup_{f\in\mcal{F}_{r_n}}\left|\frac{1}{\sqrt{n}}\sum_{i=1}^n\xi_i\epsilon_i(f(\mbf{x}_i)-f_0(\mbf{x}_i))\right|\right]\\
	& \leq\mbb{E}_\xi\left[\left|\frac{1}{n}\sum_{i=1}^n\xi_i\epsilon_i(f_n^*(\mbf{x}_i)-f_0(\mbf{x}_i))\right|\right]+K\int_0^\infty\sqrt{\frac{\log N\left(\frac{1}{2}\eta,\mcal{F}_{r_n},d\right)}{n}}\mrm{d}\eta\\
	& =\mbb{E}_\xi\left[\left|\frac{1}{n}\sum_{i=1}^n\xi_i\epsilon_i(f_n^*(\mbf{x}_i)-f_0(\mbf{x}_i))\right|\right]+K\int_0^{2V_n}\sqrt{\frac{\log N\left(\frac{1}{2}\eta,\mcal{F}_{r_n},d\right)}{n}}\mrm{d}\eta\\
	& \leq \mbb{E}_\xi\left[\left|\frac{1}{n}\sum_{i=1}^n\xi_i\epsilon_i(f_n^*(\mbf{x}_i)-f_0(\mbf{x}_i))\right|\right]+K\int_0^{2V_n}\sqrt{\frac{\log N\left(\frac{1}{2\sqrt{\sigma^2+1}}\eta,\mcal{F}_{r_n},\|\cdot\|_\infty\right)}{n}}\mrm{d}\eta,
\end{align*}
where the second equality follows from Proposition \ref{Prop: Envelope}. For $f,g\in\mcal{F}_{r_n}$,
\begin{align*}
d(f,g) & =\left[\sum_{i=1}^n\left(\frac{1}{\sqrt{n}}\epsilon_i(f(\mbf{x}_i)-f_0(\mbf{x}_i))-\frac{1}{\sqrt{n}}\epsilon_i(g(\mbf{x}_i)-f_0(\mbf{x}_i))\right)^2\right]^{1/2}\\
	& =\left(\frac{1}{n}\sum_{i=1}^n\epsilon_i^2(f(\mbf{x}_i)-g(\mbf{x}_i))^2\right)^{1/2}
\end{align*}
so that the last inequality follows by noting that
$$
d(f,g)\leq\|f-g\|_\infty\left(\frac{1}{n}\sum_{i=1}^n\epsilon_i^2\right)^{1/2}.
$$
We then evaluate these two terms. For the first term, for $n\geq N_1$, by Cauchy-Schwarz inequality, we have
\begin{align*}
\mbb{E}_\xi\left[\left|\frac{1}{n}\sum_{i=1}^n\xi_i\epsilon_i(f_n^*(\mbf{x}_i)-f_0(\mbf{x}_i))\right|\right] & \leq\frac{1}{n}\sum_{i=1}^n|\epsilon_i||f_n^*(\mbf{x}_i)-f_0(\mbf{x}_i)|\\
	& \leq\left(\frac{1}{n}\sum_{i=1}^n\epsilon_i^2\right)^{1/2}\left(\frac{1}{n}\sum_{i=1}^n(f_n^*(\mbf{x}_i)-f_0(\mbf{x}_i))^2\right)^{1/2}\\
	& \leq\sqrt{\sigma^2+1}\sup_{\mbf{x}\in\mcal{X}}|f_n^*(\mbf{x})-f_0(\mbf{x})|,\mrm{ a.s.}
\end{align*}
By choosing $f_n^*=\pi_{r_n}f_0$ and using the universal approximation theorem introduced by \citet{hornik1989multilayer}, we know that $\sup_{x\in\mcal{X}}|f_n^*(\mbf{x}_i)-f_0(\mbf{x}_i)|\to0$ as $n\to\infty$. Therefore, for any $\zeta>0$, there exists $N_2>0$, such that for all $n\geq N_2$,
$$
\sup_{x\in\mcal{X}}|f_n^*(\mbf{x}_i)-f_0(\mbf{x}_i)|<\frac{\zeta}{\sqrt{\sigma^2+1}}.
$$ 
By choosing $n\geq N_1\vee N_2$, we get
$$
\mbb{E}_\xi\left[\left|\frac{1}{n}\sum_{i=1}^n\xi_i\epsilon_i(f_n^*(\mbf{x}_i)-f_0(\mbf{x}_i))\right|\right]<\zeta\mrm{ a.s.}
$$
For the second term, we use the same bound from Theorem 14.5 in \citet{anthony2009neural} as we did in the proof of Lemma \ref{Lm: NNSieveULLN}: 
\begin{align*}
N\left(\frac{1}{2\sqrt{\sigma^2+1}}\eta,\mcal{F}_{r_n},\|\cdot\|_\infty\right) & \leq\left(\frac{8\sqrt{\sigma^2+1}e[r_n(d+2)+1]\left(\frac{1}{4}V_n\right)^2}{\eta\left(\frac{1}{4}V_n-1\right)}\right)^{r_n(d+2)+1}\\
	& :=\tilde{B}_{r_n,d,V_n}\eta^{-[r_n(d+2)+1]},
\end{align*}
where $\tilde{B}_{r_n,d,V_n}=\left(2\sqrt{\sigma^2+1}e[r_n(d+2)+1]V_n^2/(V_n-4)\right)^{r_n(d+2)+1}$. Let
\begin{align*}
B_{r_n,d,V_n} & =\log\tilde{B}_{r_n,d,V_n}-[r_n(d+2)+1]\\
	& =[r_n(d+2)+1]\left(\log\frac{2\sqrt{\sigma^2+1}e[r_n(d+2)+1]V_n^2}{V_n-4}-1\right)\\
	& =[r_n(d+2)+1]\left(\log\frac{[r_n(d+2)+1]V_n^2}{V_n-4}+\log(2\sqrt{\sigma^2+1})\right)\\
	& \leq2[r_n(d+2)+1]\log\frac{[r_n(d+2)+1]V_n^2}{V_n-4}, \mrm{ for all }n\geq N_1\vee N_3,
\end{align*}
where $N_3$ is chosen to satisfy $r_n(d+2)+1\geq2\sqrt{\sigma^2+1}$. The last inequality then follows by noting that $V_n^2-V_n+4\geq0$ for all $V_n$ so that $\log\frac{[r_n(d+1)+1]V_n^2}{V_n-4}\geq\log\frac{2\sqrt{\sigma^2+1}(V_n-4)}{V_n-4}=\log(2\sqrt{\sigma^2+1})$. We also have 
\begin{align*}
H\left(\frac{1}{2\sqrt{\sigma^2+1}}\eta,\mcal{F}_{r_n},\|\cdot\|_\infty\right) & =\log N\left(\frac{1}{2\sqrt{\sigma^2+1}}\eta,\mcal{F}_{r_n},\|\cdot\|_\infty\right)\\
	& =\log\tilde{B}_{r_n,d,V_n}+[r_n(d+2)+1]\log\frac{1}{\eta}\\
	& \leq B_{r_n,d,V_n}+[r_n(d+2)+1]\frac{1}{\eta}\\
	& \leq B_{r_n,d,V_n}\left(1+\frac{1}{\eta}\right),
\end{align*}
and hence for all $n\geq N_1\vee N_3$,
\begin{align*}
\int_0^{2V_n} & H^{1/2}\left(\frac{1}{2\sqrt{\sigma^2+1}}\eta,\mcal{F}_{r_n},\|\cdot\|_\infty\right)\mrm{d}\eta \\
	& \leq B_{r_n,d,V_n}^{1/2}\int_0^{2V_n}\left(1+\frac{1}{\eta}\right)^{1/2}\mrm{d}\eta\\
	& =B_{r_n,d,V_n}^{1/2}\left[\int_0^1\left(1+\frac{1}{\eta}\right)^{1/2}\mrm{d}\eta+\int_1^{2V_n}\left(1+\frac{1}{\eta}\right)^{1/2}\mrm{d}\eta\right]\\
	& \leq B_{r_n,d,V_n}^{1/2}\left[\sqrt{2}\int_0^1\eta^{-1/2}\mrm{d}\eta+\sqrt{2}(2V_n-1)\right]\\
	& \leq 4\sqrt{2}B_{r_n,d,V_n}^{1/2}V_n,
\end{align*}
which implies that
\begin{align*}
\int_0^{2V_n}\sqrt{\frac{H\left(\frac{1}{2\sqrt{\sigma^2+1}}\eta,\mcal{F}_{r_n},\|\cdot\|_\infty\right)}{n}}\mrm{d}\eta & \leq4\sqrt{2}n^{-1/2}B_{r_n,d,V_n}^{1/2}V_n\\
	& \sim8\sqrt{\frac{[r_n(d+2)+1]V_n^2\log(V_n[r_n(d+2)+1])}{n}},
\end{align*}
where the last part follows by noting that $\log\frac{V_n^2}{V_n-4}\sim\log V_n$. Under the assumption given in the Lemma, there exists $N_4>0$, such that for all $n\geq N_4$, we have
$$
\sqrt{\frac{[r_n(d+2)+1]V_n^2\log(V_n[r_n(d+2)+1])}{n}}<\frac{\zeta}{8}.
$$
Therefore, by choosing $n\geq N_1\vee N_2\vee N_3\vee N_4$, we get
$$
\mbb{E}_\xi\left[\sup_{f\in\mcal{F}_{r_n}}\left|\frac{1}{n}\sum_{i=1}^n\xi_i\epsilon_i(f(\mbf{x}_i)-f_0(\mbf{x}_i))\right|\right]<2\zeta\mrm{ a.s.},
$$
i.e. $\mbb{E}_\xi\left[\sup_{f\in\mcal{F}_{r_n}}\left|\frac{1}{n}\sum_{i=1}^n\xi_i\epsilon_i(f(\mbf{x}_i)-f_0(\mbf{x}_i))\right|\right]\to0$ a.s.. Moreover, based on what we have shown, for a sufficiently large $n$, we have
\begin{align*}
\mbb{E}_\xi\left[\sup_{f\in\mcal{F}_{r_n}}\left|\frac{1}{n}\sum_{i=1}^n\xi_i\epsilon_i(f(\mbf{x}_i)-f_0(\mbf{x}_i))\right|\right]\leq & \sqrt{\sigma^2+1}\|\pi_{r_n}f_0-f_0\|_\infty+\\
	& 4\sqrt{2}KB_{r_n,d,V_n}^{1/2}n^{-1/2}V_n\to0,\mrm{ a s.}.
\end{align*}
Since $\mbb{E}_\epsilon\left[\sqrt{\sigma^2+1}\|\pi_{r_n}f_0-f_0\|_\infty+4\sqrt{2}KB_{r_n,d,V_n}^{1/2}n^{-1/2}V_n\right]=\sqrt{\sigma^2+1}\|\pi_{r_n}f_0-f_0\|_\infty+4\sqrt{2}KB_{r_n,d,V_n}^{1/2}n^{-1/2}V_n\to0<\infty$, by using the Generalized Dominated Convergence Theorem, we know that 
\begin{align*}
\mbb{E}^* & \left[\sup_{f\in\mcal{F}_{r_n}}\left|\frac{1}{n}\sum_{i=1}^n\epsilon_i(f(\mbf{x}_i)-f_0(\mbf{x}_i))\right|\right]\\
 & \leq2\mbb{E}_\epsilon\mbb{E}_\xi\left[\sup_{f\in\mcal{F}_{r_n}}\left|\frac{1}{n}\sum_{i=1}^n\xi_i\epsilon_i\left(f(\mbf{x}_i)-f_0(\mbf{x}_i)\right)\right|\right]\to0,
\end{align*}
which completes the proof.
\end{proof}

Based on the above lemmas, we are ready to state the theorem on the consistency of neural network sieve estimators.

\begin{thm}\label{Thm: NNSieveConsistency}
Under the notation given above, if 
\begin{equation}\label{Eq: RateCondition}
[r_n(d+2)+1]V_n^2\log(V_n[r_n(d+2)+1]=o(n),\mrm{ as }n\to\infty,
\end{equation}
then
$$
\|\hat{f}_n-f_0\|_n\xrightarrow{p}0.
$$
\end{thm}

\begin{proof}
Since $Q$ is continuous at $f_0\in\mcal{F}$ and $Q(f_0)=\sigma^2<\infty$, for any $\epsilon>0$, we have
$$
\inf_{f:\|f-f_0\|_n\geq\epsilon}Q_n(f)-Q_n(f_0)=\inf_{f:\|f-f_0\|_n\geq\epsilon}\frac{1}{n}\sum_{i=1}^n(f(\mbf{x}_i)-f_0(\mbf{x}_i))^2\geq\epsilon^2>0.
$$
Hence, based on Lemma \ref{Lm: Compactness of NN Sieve Class}, Lemma \ref{Lm: NNSieveULLN-2} and Corollary 2.6 in \citet{white1991some}, we have
$$
\|\hat{f}_n-f_0\|_n\xrightarrow{p}0.
$$
\end{proof}

\begin{remark}\label{Rmk: Some Simple Case for Consistency}
We discuss the condition (\ref{Eq: RateCondition}) in Theorem \ref{Thm: NNSieveConsistency} via some simple examples here. If $\alpha_j=\mcal{O}(1)$ for $j=1,\ldots,r_n$, then $V_n=\mcal{O}(r_n)$ and
$$
[r_n(d+2)+1]V_n^2\log(V_n[r_n(d+2)+1])=\mcal{O}(r_n^3\log r_n).
$$
Therefore, a possible growth rate for the number of hidden units in a neural network is $r_n=o\left((n/\log n)^{1/3}\right)$. On the other hand, if we have a slow growth rate for the number of hidden units in the neural network, such as $r_n=\log V_n$, then we have
$$
[r_n(d+2)+1]V_n^2\log(V_n[r_n(d+2)+1])=\mcal{O}((V_n\log V_n)^2).
$$ 
Hence, a possible growth rate for the upper bound of the weights from the hidden layer to the output layer is $V_n=o\left(n^{1/2}/\log n\right)$.
\end{remark}

\section{Rate of Convergence}
To obtain the rate of convergence for neural network sieves, we apply Theorem 3.4.1 in van der Vaart and Wellner (1996)\cite{van1996weak}.

\begin{lemma}\label{Lm: Theorem 3.4.1 vdvw1996 Condition 1}
Let $f_n^*=\pi_{r_n}f_0\in\mcal{F}_{r_n}$. Given the above notations, for every $n$ and $\delta>8\|f_n^*-f_0\|_n$, we have
$$
\sup_{\frac{\delta}{2}<\|f-f_n^*\|_n\leq\delta, f\in\mcal{F}_{r_n}}Q_n(f_n^*)-Q_n(f)\lesssim-\delta^2.
$$
\end{lemma}

\begin{proof}
Note that
\begin{align*}
Q_n(f_n^*)-Q_n(f) & =\frac{1}{n}\sum_{i=1}^n(f_n^*(\mbf{x}_i)-f_0(\mbf{x}_i))^2+\sigma^2-\frac{1}{n}\sum_{i=1}^n(f(\mbf{x}_i)-f_0(\mbf{x}_i))^2-\sigma^2\\
	& =\|f_n^*-f_0\|_n^2-\|f-f_0\|_n^2.
\end{align*}
In order to show the result, we first provide an upper bound for $Q_n(f_n^*)-Q_n(f)$ in terms of $\|f-f_n^*\|_n$. Due to the fact that $\|\cdot\|_n$ is a pseudo-norm, the triangle inequality gives
\begin{align*}
\|f-f_n^*\|_n & \leq\|f-f_0\|_n+\|f_n^*-f_0\|_n\\
	& =\|f-f_0\|_n-\|f_n^*-f_0\|_n+2\|f_n^*-f_0\|_n.
\end{align*}
Therefore, we have
$$
\|f-f_0\|_n-\|f_n^*-f_0\|_n\geq\|f-f_n^*\|-2\|f_n^*-f_0\|_n,
$$
so that for every $f$ satisfying $\|f-f_n^*\|_n^2\geq 16\|f_n^*-f_0\|_n^2$, i.e., $\|f-f_n^*\|_n\geq4\|f_n^*-f_0\|_n$, we have
$$
\|f-f_0\|_n-\|f_n^*-f_0\|_n\geq\|f-f_n^*\|_n-\frac{1}{2}\|f-f_n^*\|_n=\frac{1}{2}\|f-f_n^*\|_n\geq0.
$$
By squaring both sides, we obtain
\begin{align*}
\frac{1}{4}\|f-f_n^*\|_n^2 & \leq\|f-f_0\|_n^2+\|f_n^*-f_0\|_n^2-2\|f-f_0\|_n\|f_n^*-f_0\|_n\\
	& \leq\|f-f_0\|_n^2+\|f_n^*-f_0\|_n^2-2\|f_n^*-f_0\|_n^2\\
	& =\|f-f_0\|_n^2-\|f_n^*-f_0\|_n^2,
\end{align*}
and hence
\begin{align*}
\sup_{\frac{\delta}{2}<\|f-f_n^*\|_n\leq\delta, f\in\mcal{F}_{r_n}}Q_n(f_n^*)-Q_n(f) & \leq\sup_{\|f-f_n^*\|_n>\frac{\delta}{2},f\in\mcal{F}_{r_n}}\|f_n^*-f_0\|_n^2-\|f-f_0\|_n^2\\
	& \leq\sup_{\|f-f_n^*\|_n>\frac{\delta}{2},f\in\mcal{F}_{r_n}}\left(-\frac{1}{4}\|f-f_n^*\|_n^2\right)\\
	& \lesssim-\delta^2.
\end{align*}
\end{proof}

\begin{lemma}\label{Lm: Theorem 3.4.1 vdvw1996 Condition 2}
For every sufficiently large $n$ and $\delta>8\|f_n^*-f_0\|_n$, we have
$$
\mbb{E}^*\left[\sup_{\frac{\delta}{2}<\|f-f_n^*\|_n\leq\delta, f\in\mcal{F}_{r_n}}\sqrt{n}\left[(\mbb{Q}_n-Q_n)(f_n^*)-(\mbb{Q}_n-Q_n)(f)\right]^+\right]\lesssim\int_0^\delta\sqrt{\log N(\eta,\mcal{F}_{r_n},\|\cdot\|_\infty)}\mrm{d}\eta
$$
\end{lemma}

\begin{proof}
Note that
\begin{align*}
(\mbb{Q}_n-Q_n)(f_n^*) & =\frac{1}{n}\sum_{i=1}^n\epsilon_i^2-\sigma^2-\frac{2}{n}\sum_{i=1}^n\epsilon_i(f_n^*(\mbf{x}_i)-f_0(\mbf{x}_i))\\
(\mbb{Q}_n-Q_n)(f_n^*) & =\frac{1}{n}\sum_{i=1}^n\epsilon_i^2-\sigma^2-\frac{2}{n}\sum_{i=1}^n\epsilon_i(f(\mbf{x}_i)-f_0(\mbf{x}_i)),
\end{align*}
we have
$$
(\mbb{Q}_n-Q_n)(f_n^*)-(\mbb{Q}_n-Q_n)(f)=\frac{2}{n}\sum_{i=1}^n\epsilon_i(f(\mbf{x}_i)-f_n^*(\mbf{x}_i)).
$$
Hence, by using the similar arguments as in the proof of Lemma \ref{Lm: NNSieveULLN-2} and applying Corollary 2.2.8 in van der Vaart and Wellner (1996)\cite{van1996weak}, we have
\begin{align*}
 & \mbb{E}^*\left[\sup_{\frac{\delta}{2}<\|f-f_n^*\|_n\leq\delta, f\in\mcal{F}_{r_n}}\sqrt{n}\left[(\mbb{Q}_n-Q_n)(f_n^*)-(\mbb{Q}_n-Q_n)(f)\right]^+\right]\\
 \leq & \mbb{E}^*\left[\sup_{\frac{\delta}{2}<\|f-f_n^*\|_n\leq\delta, f\in\mcal{F}_{r_n}}\left|\frac{1}{\sqrt{n}}\sum_{i=1}^n\epsilon_i(f(\mbf{x}_i)-f_n^*(\mbf{x}_i))\right|\right]\\ 
 \leq & 2\mbb{E}_\epsilon\mbb{E}_\xi\left[\sup_{\frac{\delta}{2}<\|f-f_n^*\|_n\leq\delta,f\in\mcal{F}_{r_n}}\left|\frac{1}{\sqrt{n}}\sum_{i=1}^n\xi_i\epsilon_i\left(f(\mbf{x}_i)-f_n^*(\mbf{x}_i)\right)\right|\right]\\
 \lesssim & \int_0^\delta\sqrt{\log N(\eta,\mcal{F}_{r_n},\|\cdot\|_\infty)}\mrm{d}\eta,
\end{align*}
where the last inequality follows since $f_n^*\in\mcal{F}_{r_n}$ for a large enough $n$.
\end{proof}

Now we are ready to apply Theorem 3.4.1 in \citet{van1996weak} to obtain the rate of convergence for neural network sieve estimators.
\begin{thm}\label{Thm: Rate of Convergence}
Based on the above notations, if 
$$
\eta_n=\mcal{O}\left(\min\{\|\pi_{r_n}f_0-f_0\|_n^{2},r_n(d+2)\log(r_nV_n(d+2))/n, r_n(d+2)\log n/n\}\right),
$$
then
$$
\|\hat{f}_n-f_0\|_n=\mcal{O}_p\left(\max\left\{\|\pi_{r_n}f_0-f_0\|_n,\sqrt{\frac{r_n(d+2)\log[r_nV_n(d+2)]}{n}},\sqrt{\frac{r_n(d+2)\log n}{n}}\right\}\right).
$$
\end{thm}

\begin{proof}
Use the same bound from Theorem 14.5 in \citet{anthony2009neural}, we have
\begin{align*}
\log N(\eta,\mcal{F}_{r_n},\|\cdot\|_n) & \leq\log N(\eta,\mcal{F}_{r_n},\|\cdot\|_\infty)\\
	& \leq\log\left(\frac{4e[r_n(d+2)+1]\left(\frac{1}{4}V_n\right)^2}{\eta\left(\frac{1}{4}V_n-1\right)}\right)^{r_n(d+2)+1}\\
	& =[r_n(d+2)+1]\log\frac{\tilde{C}_{r_n,d,V_n}}{\eta},
\end{align*}
where $\tilde{C}_{r_n,d,V_n}=\frac{e[r_n(d+2)+1]V_n^2}{V_n-4}>e$. Then from Lemma 3.8 in \citet{mendelson2003few}, for $\delta<1$,
\begin{align*}
\int_0^\delta\sqrt{\log N(\eta, \mcal{F}_{r_n}, \|\cdot\|_n)}\mrm{d}\eta & \leq[r_n(d+2)+1]^{1/2}\int_0^\delta\sqrt{\log\frac{\tilde{C}_{r_n,d,V_n}}{\eta}}\mrm{d}\eta\\
	& \lesssim[r_n(d+2)+1]^{1/2}\delta\sqrt{\log\frac{\tilde{C}_{r_n,d,V_n}}{\delta}}\\
	& :=\phi_n(\delta).
\end{align*}
Define $h:\delta\mapsto\phi_n(\delta)/\delta^\alpha=[r_n(d+2)+1]^{1/2}\delta^{1-\alpha}\sqrt{\log\frac{\tilde{C}_{r_n,d,V_n}}{\delta}}$. Since for $0<\delta<1$ and $1<\alpha<2$
\begin{align*}
h'(\delta) & =[r_n(d+2)+1]^{1/2}\left((1-\alpha)\delta^{-\alpha}\sqrt{\log\frac{\tilde{C}_{r_n,d,V_n}}{\delta}}-\frac{1}{2}\frac{\delta^2}{\tilde{C}_{r_n,d,V_n}}\frac{\tilde{C}_{r_n,d,V_n}}{\delta^2}\log^{-1/2}\frac{\tilde{C}_{r_n,d,V_n}}{\delta}\right)\\
	& =[r_n(d+2)+1]^{1/2}\left((1-\alpha)\delta^{-\alpha}\sqrt{\log\frac{\tilde{C}_{r_n,d,V_n}}{\delta}}-\frac{1}{2}\log^{-1/2}\frac{\tilde{C}_{r_n,d,V_n}}{\delta}\right)\\
	& <0,
\end{align*}
$\delta\mapsto\phi_n(\delta)/\delta^\alpha$ is decreasing on $(0,\infty)$. Let $\rho_n\lesssim\|\pi_{r_n}f_0-f_0\|_n^{-1}$. Note that
\begin{align*}
\rho_n^2\phi_n\left(\frac{1}{\rho_n}\right) & =\rho_n[r_n(d+2)+1]^{1/2}\log^{1/2}\left(\rho_n\tilde{C}_{r_n,d,V_n}\right)\\
	& =[r_n(d+2)+1]^{1/2}\rho_n\sqrt{\log\rho_n+\log\tilde{C}_{r_n,d,V_n}}
\end{align*}
and
\begin{align*}
\log\tilde{C}_{r_n,d,V_n} & =1+\log\frac{[r_n(d+2)+1]V_n^2}{V_n-4}\lesssim\log\frac{[r_n(d+2)+1]V_n^2}{V_n-4}\\
	& \sim\log[r_nV_n(d+2)],
\end{align*}
we have
\begin{align*}
\rho_n^2\phi_n\left(\frac{1}{\rho_n}\right)\lesssim\sqrt{n} & \Leftrightarrow r_n(d+2)\rho_n^2\left(\log\rho_n+\log[r_nV_n(d+2)]\right)\lesssim n.
\end{align*}
Therefore, for
$$
\rho_n\lesssim\min\left\{\left(\frac{n}{r_n(d+2)\log[r_nV_n(d+2)]}\right)^{1/2},\left(\frac{n}{r_n(d+2)\log n}\right)^{1/2}\right\},
$$
we have $\rho_n^2\phi_n\left(\frac{1}{\rho_n}\right)\lesssim\sqrt{n}$. Based on these observation, Lemma \ref{Lm: Theorem 3.4.1 vdvw1996 Condition 1}, Lemma \ref{Lm: Theorem 3.4.1 vdvw1996 Condition 2} and Theorem 3.4.1 in van der Vaart and Wellner (1996)\cite{van1996weak} imply that
$$
\|\hat{f}_n-\pi_{r_n}f_0\|_n=\mcal{O}_p\left(\max\left\{\|\pi_{r_n}f_0-f_0\|_n,\sqrt{\frac{r_n(d+2)\log[r_nV_n(d+2)]}{n}},\sqrt{\frac{r_n(d+2)\log n}{n}}\right\}\right).
$$
By using the triangle inequality, we can further get
\begin{align*}
\|\hat{f}_n-f_0\|_n & \leq\|\hat{f}_n-\pi_{r_n}f_0\|_n+\|\pi_{r_n}f_0-f_0\|_n\\
	& =\mcal{O}_p\left(\max\left\{\|\pi_{r_n}f_0-f_0\|_n,\sqrt{\frac{r_n(d+2)\log[r_nV_n(d+2)]}{n}},\sqrt{\frac{r_n(d+2)\log n}{n}}\right\}\right).
\end{align*}
\end{proof}

\begin{remark}
Recall that a sufficient condition to ensure consistency is $r_n(d+2)V_n^2\log[r_nV_n(d+2)]=o(n)$. Under such a condition, $r_n(d+2)\log[r_nV_n(d+2)]\leq n$, the rate of convergence can be simplified to
$$
\|\hat{f}_n-f_0\|_n=\mcal{O}_p\left(\max\left\{\|\pi_{r_n}f_0-f_0\|_n,\sqrt{\frac{r_n(d+2)\log n}{n}}\right\}\right).
$$
\end{remark}

If we assume $f_0\in\mcal{F}$ where $\mcal{F}$ is the space of functions with finite first absolute moments of the Fourier magnitude distributions, i.e.,
\begin{align*}
\mcal{F}= & \left\{f:\mbb{R}^d\to\mbb{R}:f(\mbf{x})=\int\exp\left\{i\mbf{a}^T\mbf{x}\right\}\mrm{d}\mu_f(\mbf{a}),\right.\\
	& \hspace{1.5cm}\left.\|\mu_f\|_1:=\int\max(\|\mbf{a}\|_1, 1)\mrm{d}|\mu_f|(\mbf{a})\leq C\right\},\numberthis\label{Eq: Finite first moment abosolute moments of Fourier...}
\end{align*}
where $\mu_f$ is a complex measure on $\mbb{R}^d$. $|\mu_f|$ denotes the total variation of $\mu_f$, i.e., $|\mu|(A)=\sup\sum_{n=1}^\infty|\mu(A_n)|$ and the supremum is taken over all measurable partitions $\{A_n\}_{n=1}^\infty$ of $A$. $\|\mbf{a}\|_1=\sum_{i=1}^d|a_i|$ for $\mbf{a}=[a_1,\ldots,a_d]^T\in\mbb{R}^d$. Theorem 3 in \citet{makovoz1996random} shows that $\delta_n:=\|f_0-\pi_{r_n}f_0\|_n\lesssim r_n^{-1/2-1/(2d)}$. Therefore, if we let $d$ fixed and $\rho_n=\delta_n^{-1}$ and $V_n\equiv V$ in the proof of Theorem \ref{Thm: Rate of Convergence}, $\delta_n$ must also satisfy the following inequality:
\begin{align*}
 & \rho_n^2\phi\left(\frac{1}{\rho_n}\right)\lesssim\rho_nr_n^{1/2}\log^{1/2}\left(\rho_n\tilde{C}_{r_n,d,V_n}\right)\lesssim\sqrt{n}\\
\Rightarrow\quad & \rho_n^2r_n\log\rho_n+\rho_n^2r_n\log r_n\lesssim n\\
\Rightarrow\quad & \delta_n^{-2}\left(-r_n\log\delta_n+r_n\log r_n\right)\lesssim n\\
\Rightarrow\quad & r_n^{1+\frac{1}{d}}r_n\log r_n\lesssim n.
\end{align*}
One possible choice of $r_n$ to satisfy such condition is $r_n\asymp(n/\log n)^{\frac{d}{2+d}}$. In such a case, we obtain
$$
\|\hat{f}_n-f_0\|_n=\mcal{O}_p\left(\left(\frac{n}{\log n}\right)^{-\frac{1+1/d}{4(1+1/(2d))}}\right),
$$ 
which is the same rate obtained in \citet{chen1998sieve}. It is interesting to note that in the case where $d=1$, we have $\|\hat{f}_n-f_0\|_n=\mcal{O}_p\left((n/\log n)^{-1/3}\right)$. Such rate is close to the $\mcal{O}_p(n^{-1/3})$, which is the convergence rate in non-parametric least square problems when the class of functions considered has bounded variation in $\mbb{R}$ (see Example 9.3.3 in \citet{van2000empirical}). As shown in Proposition \ref{Prop: BoundedVariation} in the Appendix, $\mcal{F}_{r_n}$ is a class of functions with bounded variation in $\mbb{R}$. Therfore, the convergence rate we obtained makes sense.

\section{Asymptotic Normality}\label{Sec: Asymptotic Normality}
To establish the asymptotic normality of sieve estimator for neural network, we follow the idea in \citet{shen1997methods} and start by calculating the G\^ateaux derivative of the empirical criterion function $\mbb{Q}_n(f)=n^{-1}\sum_{i=1}^n(y_i-f(\mbf{x}_i))^2$,
\begin{align*}
\mbb{Q}_{n,f_0}'[f-f_0] & =\lim_{t\to0}\frac{1}{t}\left[\frac{1}{n}\sum_{i=1}^n(y_i-f_0(\mbf{x}_i)-t(f(\mbf{x}_i)-f_0(\mbf{x}_i)))^2-\frac{1}{n}\sum_{i=1}^n(y_i-f_0(\mbf{x}_i))^2\right]\\
	& =\lim_{t\to0}\frac{1}{n}\sum_{i=1}^n\frac{1}{t}\left[(y_i-f_0(\mbf{x}_i))^2-2t(y_i-f_0(\mbf{x}_i))(f(\mbf{x}_i)-f_0(\mbf{x}_i))\right.\\
	& \hspace{5cm}\left.+t^2(f(\mbf{x}_i)-f_0(\mbf{x}_i))^2-(y_i-f_0(\mbf{x}_i))^2\right]\\
	& =-\frac{2}{n}\sum_{i=1}^n\epsilon_i(f(\mbf{x}_i)-f_0(\mbf{x}_i)).
\end{align*}
Then the remainder of first-order functional Taylor series expansion is
\begin{align*}
R_n[f-f_0] & =\mbb{Q}_n(f)-\mbb{Q}_n(f_0)-\mbb{Q}_{n,f_0}'[f-f_0]\\
	& =\frac{1}{n}\sum_{i=1}^n(y_i-f(\mbf{x}_i))^2-\frac{1}{n}\sum_{i=1}^n(y_i-f(\mbf{x}_i))^2+\frac{2}{n}\sum_{i=1}^n\epsilon_i(f(\mbf{x}_i)-f_0(\mbf{x}_i))\\
	& =\frac{1}{n}\sum_{i=1}^n(\epsilon_i+f_0(\mbf{x}_i)-f(\mbf{x}_i))^2-\frac{1}{n}\sum_{i=1}^n\epsilon_i^2+\frac{2}{n}\sum_{i=1}^n\epsilon_i(f(\mbf{x}_i)-f_0(\mbf{x}_i))\\
	& =\frac{1}{n}\sum_{i=1}^n(f(\mbf{x}_i)-f_0(\mbf{x}_i))^2\\
	& =\|f-f_0\|_n^2.
\end{align*}

As will be seen in the proof of asymptotic normality, the rate of convergence for the empirical process $\{n^{-1/2}\sum_{i=1}^n\epsilon_i(f(\mbf{x}_i)-f_0(\mbf{x}_i)):f\in\mcal{F}_{r_n}\}$ plays an important role. Here we establish a lemma, which will be used to find the desired rate of convergence.

\begin{lemma}\label{Lm: LargeDevEmpProcess}
Let $X_1,\ldots,X_n$ be independent random variables with $X_i\sim P_i$. Define the empirical process $\{\nu_n(f)\}$ as
$$
\nu_n(f)=\frac{1}{\sqrt{n}}\sum_{i=1}^n[f(X_i)-P_if].
$$
Let $\mcal{F}_n=\{f:\|f\|_\infty\leq V_n\}$, $\epsilon>0$ and $\alpha\geq\sup_{f\in\mcal{F}_n}n^{-1}\sum_{i=1}^n\mrm{Var}[f(X_i)]$ be arbitrary. Define $t_0$ by $H(t_0,\mcal{F}_n,\|\cdot\|_\infty)=\frac{\epsilon}{4}\psi(M,n,\alpha)$, where $\psi(M,n,\alpha)=M^2/\left[2\alpha\left(1+\frac{MV_n}{2\sqrt{n}\alpha}\right)\right]$. If
\begin{equation}\label{Eq: entropyBound}
H(u,\mcal{F}_n,\|\cdot\|_\infty)\leq A_nu^{-r},
\end{equation}
for some $0<r<2$ and $u\in(0,a]$, where $a$ is a small positive number, and there exists a positive constant $K_i=K_i(r,\epsilon)$, $i=1,2$ such that
$$
M\geq K_1A_n^{\frac{2}{r+2}}V_n^{\frac{2-r}{r+2}}n^{\frac{r-2}{2(r+2)}}\vee K_2A_n^{1/2}\alpha^{\frac{2-r}{4}},
$$
we have
$$
\mbb{P}^*\left(\sup_{f\in\mcal{F}_n}|\nu_n(f)|>M\right)\leq5\exp\left\{-(1-\epsilon)\psi(M,n,\alpha)\right\}.
$$
\end{lemma}

\begin{proof}
The proof of the lemma is similar to the proof of Corollary 2.2 in \citet{alexander1984probability} and the proof of Lemma 1 in \citet{shen1994convergence}. Since $H(u,\mcal{F}_n,\|\cdot\|_\infty)\leq A_nu^{-r}$ for some $0<r<2$, we have
$$
I(s,t):=\int_s^tH^{1/2}(u,\mcal{F}_n,\|\cdot\|_\infty)\mrm{d}u\leq2(2-r)^{-1}A_n^{\frac{1}{2}}t^{1-\frac{r}{2}}.
$$
Based on the assumption of 
$$
A_nt_0^{-r}\geq H(t_0,\mcal{F}_n,\|\cdot\|_\infty)=\frac{\epsilon}{4}\psi(M,n,\alpha), 
$$
we have $t_0\leq\left[\frac{4A_n}{\epsilon\psi}\right]^{1/r}$.
Note that $\psi(M,n,\alpha)\geq M^2/(4\alpha)$ if $M\leq3\sqrt{n}\alpha/V_n$ and $2(\sqrt{n}\alpha+MV_n/3)\leq4MV_n/3$ if $M\geq3\sqrt{n}\alpha/V_n$ and hence $\psi(M,n,\alpha)\geq3\sqrt{n}M/(4V_n)$. In summary,
$$
\psi(M,n,\alpha)\geq\left\{\begin{array}{ll}
M^2/(4\alpha) & \mrm{if }M<3\sqrt{n}\alpha/V_n,\\
3\sqrt{n}M/(4V_n)  & \mrm{if }M\geq3\sqrt{n}\alpha/V_n
\end{array}
\right..
$$
Therefore, if $M\geq3\sqrt{n}\alpha/V_n$,
\begin{align*}
2^8\epsilon^{-3/2}I\left(\frac{\epsilon M}{64\sqrt{n}}, t_0\right) & \leq 2^9\epsilon^{-3/2}(2-r)^{-1}A_n^{1/2}t_0^{1-\frac{r}{2}}\\
	& \leq 2^9\epsilon^{-3/2}(2-r)^{-1}\left(\frac{4}{\epsilon}\right)^{\frac{1}{r}-\frac{1}{2}}A_n^{1/r}\left(\frac{3}{4V_n}\sqrt{n}M\right)^{\frac{1}{2}-\frac{1}{r}}\\
	& =\tilde{K}_1A_n^{1/r}V_n^{\frac{1}{r}-\frac{1}{2}}n^{\frac{1}{4}-\frac{1}{2r}}M^{\frac{1}{2}-\frac{1}{r}},
\end{align*}
where $\tilde{K}_1=2^9\epsilon^{-3/2}(2-r)^{-1}\left(\frac{4}{\epsilon}\right)^{\frac{1}{r}-\frac{1}{2}}\left(\frac{3}{4}\right)^{\frac{1}{2}-\frac{1}{r}}$. Hence
\begin{align*}
2^8\epsilon^{-3/2}I\left(\frac{\epsilon M}{64\sqrt{n}}, t_0\right)<M & \Leftrightarrow \tilde{K}_1A_n^{1/r}V_n^{\frac{1}{r}-\frac{1}{2}}n^{\frac{1}{4}-\frac{1}{2r}}M^{\frac{1}{2}-\frac{1}{r}}<M\\
	& \Leftrightarrow\tilde{K}_1A_n^{1/r}V_n^{\frac{1}{r}-\frac{1}{2}}n^{\frac{r-2}{4r}}<M^{\frac{1}{r}+\frac{1}{2}}\\
	& \Leftrightarrow M>K_1A_n^{\frac{2}{r+2}}V_n^{\frac{2-r}{r+2}}n^{\frac{r-2}{2(r+2)}},
\end{align*}
where $K_1=\tilde{K}_1^{\frac{2r}{r+2}}$. On the other hand, if $M<3\sqrt{n}\alpha/V_n$,
\begin{align*}
2^8\epsilon^{-3/2}I\left(\frac{\epsilon M}{64\sqrt{n}}, t_0\right) & \leq 2^9\epsilon^{-3/2}(2-r)^{-1}A_n^{1/2}t_0^{1-\frac{r}{2}}\\
	& \leq 2^9\epsilon^{-3/2}(2-r)^{-1}\left(\frac{4}{\epsilon}\right)^{\frac{1}{r}-\frac{1}{2}}A_n^{1/r}\left(\frac{M^2}{4\alpha}\right)^{\frac{1}{2}-\frac{1}{r}}\\
	& =\tilde{K}_2A_n^{1/r}M^{1-\frac{2}{r}}\alpha^{\frac{1}{r}-\frac{1}{2}},
\end{align*}
where $\tilde{K}_2=2^9\epsilon^{-3/2}(2-r)^{-1}\left(\frac{4}{\epsilon}\right)^{\frac{1}{r}-\frac{1}{2}}\left(\frac{1}{4}\right)^{\frac{1}{2}-\frac{1}{r}}$. Hence
\begin{align*}
2^8\epsilon^{-3/2}I\left(\frac{\epsilon M}{64\sqrt{n}}, t_0\right)<M & \Leftrightarrow\tilde{K}_2A_n^{1/r}M^{1-\frac{2}{r}}\alpha^{\frac{1}{r}-\frac{1}{2}}<M\\
	& \Leftrightarrow\tilde{K}_2A_n^{1/r}\alpha^{\frac{2-r}{2r}}<M^{\frac{2}{r}}\\
	& \Leftrightarrow M>K_2A_n^{1/2}\alpha^{\frac{2-r}{4}},
\end{align*}
where $K_2=\tilde{K}_2^{r/2}$. In conclusion, if $M\geq K_1A_n^{\frac{2}{r+2}}V_n^{\frac{2-r}{r+2}}n^{\frac{r-2}{2(r+2)}}\vee K_2A_n^{1/2}\alpha^{\frac{2-r}{4}}$, then $2^8\epsilon^{-3/2}I\left(\frac{\epsilon M}{64\sqrt{n}}, t_0\right)<M$. By Theorem 2.1 in  \citet{alexander1984probability}, we have the desired result.
\end{proof}

As a Corollary to Lemma \ref{Lm: LargeDevEmpProcess}, we can show that the supremum of the empirical process $\{n^{-1/2}\sum_{i=1}^n\epsilon_i(f(\mbf{x}_i)-f_0(\mbf{x}_i)):f\in\mcal{F}_{r_n}\}$ converges to 0 in probability.

\begin{corollary}\label{Cor: LargeDevEmpProc}
Let $\rho_n$ satisfy $\rho_n\|\hat{f}_n-f_0\|_n=\mcal{O}_p(1)$ and $\mcal{F}_{r_n}$ be the class of neural network sieves as defined in (\ref{Eq: NNSieve}). Suppose that $\mbb{E}[|\epsilon|^{2+\lambda}]<\infty$ for some $\lambda>0$. Then under the conditions
\begin{itemize}
\item[(C1)] $r_n(d+2)V_n\log[r_nV_n(d+2)]=o(n^{1/4})$;
\item[(C2)] $n\rho_n^{-2}/V_n^\lambda=o(1)$,
\end{itemize}
we have
$$
\sup_{\|f-f_0\|_n\leq\rho_n^{-1},f\in\mcal{F}_{r_n}}\left|\frac{1}{\sqrt{n}}\sum_{i=1}^n\epsilon_i(f-f_0)(\mbf{x}_i)\right|=o_p(1).
$$
\end{corollary}

\begin{proof}
To establish the desired result, we apply the truncation device.
\begin{align*}
 & \mbb{P}^*\left(\sup_{\|f-f_0\|_n\leq\rho_n^{-1},f\in\mcal{F}_{r_n}}\left|\frac{1}{\sqrt{n}}\sum_{i=1}^n\epsilon_i(f-f_0)(\mbf{x}_i)\right|\gtrsim M\right)\\
\leq & \mbb{P}^*\left(\sup_{\|f-f_0\|_n\leq\rho_n^{-1},f\in\mcal{F}_{r_n}}\left|\frac{1}{\sqrt{n}}\sum_{i=1}^n\epsilon_i\mbb{I}_{\{|\epsilon_i|\leq V_n\}}(f-f_0)(\mbf{x}_i)\right|\gtrsim M\right)\\
 & \hspace{2cm}+\mbb{P}^*\left(\sup_{\|f-f_0\|_n\leq\rho_n^{-1},f\in\mcal{F}_{r_n}}\left|\frac{1}{\sqrt{n}}\sum_{i=1}^n\epsilon_i\mbb{I}_{\{|\epsilon_i|> V_n\}}(f-f_0)(\mbf{x}_i)\right|\gtrsim M\right)\\
:= & (I)+(II)
\end{align*}
For (I), we can apply Lemma \ref{Lm: LargeDevEmpProcess} directly. Note that $|\epsilon\mbb{I}_{\{|\epsilon|\leq V_n\}}(f-f_0)(\mbf{x})|\leq V_n(V_n+\|f_0\|_\infty)\lesssim V_n^2$ since $\|f_0\|_\infty<\infty$ and for $0<\eta<1$,
\begin{align*}
\log N(\eta,\mcal{F}_{r_n},\|\cdot\|_\infty) & \leq\log\left(\frac{4e[r_n(d+2)+1]\left(\frac{1}{4}V_n\right)^2}{\eta\left(\frac{1}{4}V_n-1\right)}\right)^{r_n(d+2)+1}\\
	& =[r_n(d+2)+1]\left(\log\tilde{C}_{r_n,d,V_n}+\log\frac{1}{\eta}\right)\\
	& \leq[r_n(d+2)+1]\left(\log\tilde{C}_{r_n,d,V_n}+\frac{1}{\eta}-1\right)\\
	& =C_{r_n,d,V_n}\left(1+\frac{1}{\eta}\right)\\
	& \leq2C_{r_n,d,V_n}\frac{1}{\eta},
\end{align*}
where $\tilde{C}_{r_n,d,V_n}=\frac{e[r_n(d+2)+1]V_n^2}{V_n-4}$ and
\begin{align*}
C_{r_n,d,V_n} & =[r_n(d+2)+1]\log\tilde{C}_{r_n,d,V_n}-[r_n(d+2)+1]\\
	& =[r_n(d+2)+1]\log\frac{[r_n(d+2)+1]V_n^2}{V_n-4}\\
	& \sim r_n(d+2)\log[r_nV_n(d+2)].
\end{align*}
Therefore, equation (\ref{Eq: entropyBound}) is satisfied with $r=1$ and $A_n=2C_{r_n,d,V_n}$. Following from Lemma \ref{Lm: LargeDevEmpProcess}, for $M\gtrsim C_{r_n,d,V_n}^{2/3}V_n^{2/3}n^{-1/6}\vee C_{r_n,d,V_n}^{1/2}\alpha^{1/4}$, we have $(I)\leq5\exp\left\{-(1-\epsilon)\psi(M,n,\alpha)\right\}$ and hence
$$
\sup_{\|f-f_0\|\leq\rho_n^{-1},f\in\mcal{F}_{r_n}}\left|\frac{1}{\sqrt{n}}\epsilon_i\mbb{I}_{\{|\epsilon_i|\leq V_n\}}(f-f_0)(\mbf{x}_i)\right|=\mcal{O}_p\left(\frac{C_{r_n,d,V_n}^{2/3}V_n^{2/3}}{n^{1/6}}\right).
$$
Since by using (C1),
$$
\frac{C_{r_n,d,V_n}^{2/3}V_n^{2/3}}{n^{1/6}}\sim\left(\frac{r_n(d+2)V_n\log[r_nV_n(d+2)]}{n^{1/4}}\right)^{2/3}=o_p(1).
$$
For (II), by using the Cauchy-Schwarz inequality, we have
$$
\left|\frac{1}{n}\sum_{i=1}^n\epsilon_i\mbb{I}_{\{|\epsilon_i|>V_n\}}(f-f_0)(\mbf{x}_i)\right|\leq\left(\frac{1}{n}\sum_{i=1}^n\epsilon_i^2\mbb{I}_{\{|\epsilon_i|>V_n\}}\right)^{1/2}\|f-f_0\|_n.
$$
Then it follows from the Markov inequality that
\begin{align*}
(II) & =\mbb{P}^*\left(\sup_{\|f-f_0\|_n\leq\rho_n^{-1},f\in\mcal{F}_{r_n}}\left|\frac{1}{n}\sum_{i=1}^n\epsilon_i\mbb{I}_{\{|\epsilon_i|> V_n\}}(f-f_0)(\mbf{x}_i)\right|\gtrsim Mn^{-1/2}\right)\\
	& \leq\mbb{P}\left(\left(\frac{1}{n}\sum_{i=1}^n\epsilon_i^2\mbb{I}_{\{|\epsilon_i|>V_n\}}\right)^{1/2}\rho_n^{-1}\gtrsim Mn^{-1/2}\right)\\
	& =\mbb{P}\left(\frac{1}{n}\sum_{i=1}^n\epsilon_i^2\mbb{I}_{\{|\epsilon_i|>V_n\}}\gtrsim M^2n^{-1}\rho_n^2\right)\\
	& \lesssim M^{-2}n\rho_n^{-2}\mbb{E}[\epsilon^2\mbb{I}_{|\epsilon|>V_n}]\\
	& \lesssim M^{-2}n\rho_n^{-2}\frac{\mbb{E}[|\epsilon|^{2+\lambda}]}{V_n^\lambda}.
\end{align*}
Based on condition (C2), we have $(II)\to0$, and
$$
\sup_{\|f-f_0\|_n\leq\rho_n^{-1},f\in\mcal{F}_{r_n}}\left|\frac{1}{n}\sum_{i=1}^n\epsilon_i\mbb{I}_{\{|\epsilon_i|> V_n\}}(f-f_0)(\mbf{x}_i)\right|=o_p(1).
$$
Combining the results we obtained above, we get
$$
\sup_{\|f-f_0\|_n\leq\rho_n^{-1},f\in\mcal{F}_{r_n}}\left|\frac{1}{\sqrt{n}}\sum_{i=1}^n\epsilon_i(f-f_0)(\mbf{x}_i)\right|=o_p(1)
$$
\end{proof}

\begin{remark}
Condition (C2) can be further simplified using the results from Theorem \ref{Thm: Rate of Convergence}. If 
$$
\eta_n=\mcal{O}\left(\min\{\|\pi_{r_n}f_0-f_0\|_n^{2},r_n(d+2)\log(r_nV_n(d+2))/n, r_n(d+2)\log n/n\}\right),
$$
then 
$$
\rho_n^{-1}\asymp\max\left\{\|\pi_{r_n}f_0-f_0\|_n,\sqrt{r_n(d+2)\log[r_nV_n(d+2)]/n},\sqrt{r_n(d+2)\log n/n}\right\}.
$$ 
It follows from condition (C1) that 
$$
\rho_n^{-1}\asymp\max\left\{\|\pi_{r_n}f_0-f_0\|_n,\sqrt{r_n(d+2)\log n/n}\right\}.
$$ 
For simplicity, we assume that $\rho_n^{-1}\asymp\sqrt{r_n(d+2)\log n/n}$, which holds for functions having finite first absolute moments of the Fourier magnitude distributions as discussed at the end of section 4.4. Then in this case, 
$$
n\rho_n^{-2}/V_n^\lambda\asymp r_n(d+2)\log n/V_n^\lambda,
$$
so that condition (C2) becomes $r_n(d+2)\log n/V_n^\lambda\to0$.
\end{remark}

Now we are going to establish the asymptotic normality for neural network estimators. For $f\in\{f\in\mcal{F}_{r_n}:\|f-f_0\|_n\leq\rho_n^{-1}\}$, we consider a local alternative 
\begin{equation}\label{Eq: LocalAlt}
\tilde{f}_n(f)=(1-\delta_n)f+\delta_n(f_0+\iota),
\end{equation}
where $0\leq\delta_n=\eta_n^{1/2}=o(n^{-1/2})$ is chosen such that $\rho_n\delta_n=o(1)$ and $\iota(\mbf{x})\equiv1$.

\begin{thm}[Asymptotic Normality]\label{Thm: Asymptotic Normality}
Suppose that $0\leq\eta_n=o(n^{-1})$ and conditions (C1) and (C2) in Corollary \ref{Cor: LargeDevEmpProc} hold. We further assume that the following two conditions hold
\begin{itemize}
\item[(C3)] $\sup_{f\in\mcal{F}_{r_n}:\|f-f_0\|_n\leq\rho_n^{-1}}\|\pi_{r_n}\tilde{f}_n(f)-\tilde{f}_n(f)\|_n=\mcal{O}_p(\rho_n\delta_n^2)$;
\item[(C4)] $\sup_{f\in\mcal{F}_{r_n}:\|f-f_0\|_n\leq\rho_n^{-1}}\frac{1}{n}\sum_{i=1}^n\epsilon_i\left(\pi_{r_n}\tilde{f}_n(f)(\mbf{x}_i)-\tilde{f}_n(f)(\mbf{x}_i)\right)=\mcal{O}_p(\delta_n^2)$,
\end{itemize} 
then
$$
\frac{1}{\sqrt{n}}\sum_{i=1}^n\left[\hat{f}_n(\mbf{x}_i)-f_0(\mbf{x}_i)\right]\xrightarrow{d}\mcal{N}(0,\sigma^2).
$$
\end{thm}

Before we proceed to the proof of the theorem, let us focus on the conditions given in the theorem. Note that if (C1) holds, we have 
$$
r_n(d+2)V_n^2\log[r_nV_n(d+2)]\leq[r_n(d+2)]^4V_n^4\left(\log[r_nV_n(d+2)]\right)^4=o(n),
$$
so it is a sufficient condition to ensure the consistency of the neural network sieve estimator. As in Remark \ref{Rmk: Some Simple Case for Consistency}, we consider some simple scenarios here. If $V_n=\mcal{O}(r_n)$, then $r_n(d+2)V_n\log[r_nV_n(d+2)]=\mcal{O}\left(r_n^2\log r_n\right)$ so that a possible growth rate for $r_n$ is $r_n=o\left(n^{1/8}/(\log n)^2\right)$. On the other hand, if $r_n=\log V_n$, then $r_n(d+2)V_n\log[r_nV_n(d+2)]=\mcal{O}\left(V_n(\log V_n)^2\right)$ and a possible growth rate for $V_n$ is $V_n=o(n^{1/4}/(\log n)^2)$. Thus, in both cases, the growth rate required for the asymptotic normality of neural network sieve estimator is slower than the growth rate required for the consistency as given in Remark \ref{Rmk: Some Simple Case for Consistency}. One explanation is that due to the Universal Approximation Theorem, a neural network with one hidden layer can approximate a continuous function on compact support arbitrarily well if the number of hidden units is sufficiently large. Therefore, if the number of hidden units is too large, the neural network sieve estimator $\hat{f}_n$ may be very close to the best projector of the true function $f_0$ in $\mcal{F}_{r_n}$ so that the error $\sum_{i=1}^n\left[\hat{f}_n(\mbf{x}_i)-f_0(\mbf{x}_i)\right]$ could be close to zero, resulting a small variation. By allowing slower growth rate of the number of hidden units can increase the variations of $\sum_{i=1}^n\left[\hat{f}_n(\mbf{x}_i)-f_0(\mbf{x}_i)\right]$, which makes the asymptotic normality more reasonable. On the other hand, condition (C3) and condition (C4) are similar conditions as in \citet{shen1997methods}, which are known for conditions on approximation error. These conditions indicate that the approximation rate of a single layer neural network cannot be too slow, otherwise it may require a huge number of samples to reach the desired approximation error. Therefore, the conditions in the theorem can be considered as a trade-off between bias and variance.

\begin{proof}[Proof of Theorem \ref{Thm: Asymptotic Normality}]
The main idea of the proof is to use the functional Taylor series expansion for $\mbb{Q}_n(f)$ and then carefully bound each term in the expansion. For any $f\in\{f\in\mcal{F}_{r_n}:\|f-f_0\|_n\leq\rho_n^{-1}\}$,
\begin{align*}
\mbb{Q}_n(f) & =\mbb{Q}_n(f_0)+\mbb{Q}_{n,f_0}'[f-f_0]+R_n[f-f_0]\\
	& =\frac{1}{n}\sum_{i=1}^n\epsilon_i^2-\frac{2}{n}\sum_{i=1}^n\epsilon_i(f(\mbf{x}_i)-f_0(\mbf{x}_i))+\frac{1}{n}\sum_{i=1}^n(f(\mbf{x}_i)-f_0(\mbf{x}_i))^2.\numberthis\label{Eq: Functional Taylor Expansion of Qn}
\end{align*}
Note that
\begin{align*}
\|\tilde{f}_n(f)-f_0\|_n & =\|(1-\delta_n)\hat{f}_n+\delta_n(f_0+\iota)-f_0\|_n\\
	& =\|(1-\delta_n)(\hat{f}_n-f_0)+\delta_n\iota\|_n\\
	& \leq(1-\delta_n)\|\hat{f}_n-f_0\|_n+\delta_n,
\end{align*}
and since $\delta_n=o(n^{-1/2})$, we can know that with probability tending to 1, $\|\tilde{f}_n(f)-f_0\|_n\leq\rho_n^{-1}$. Then replacing $f$ in (\ref{Eq: Functional Taylor Expansion of Qn}) by $\hat{f}_n$ and $\pi_{r_n}\tilde{f}_n(f)$ as defined in (\ref{Eq: LocalAlt}), we get
\begin{align*}
\mbb{Q}_n(\hat{f}_n) & =\frac{1}{n}\sum_{i=1}^n\epsilon_i^2-\frac{2}{n}\sum_{i=1}^n\epsilon_i(\hat{f}_n(\mbf{x}_i)-f_0(\mbf{x}_i))+\|\hat{f}_n-f_0\|_n^2\\
\mbb{Q}_n(\pi_{r_n}\tilde{f}_n(f)) & =\frac{1}{n}\sum_{i=1}^n\epsilon_i^2-\frac{2}{n}\sum_{i=1}^n\epsilon_i(\pi_{r_n}\tilde{f}_n(f)(\mbf{x}_i)-f_0(\mbf{x}_i))+\|\pi_{r_n}\tilde{f}_n(f)-f_0\|_n^2.
\end{align*}
Subtracting these two equations yields
\begin{align*}
\mbb{Q}_n(\hat{f}_n) & =\mbb{Q}_n(\pi_{r_n}\tilde{f}_n(f))+\frac{2}{n}\sum_{i=1}^n\epsilon_i\left(\pi_{r_n}\tilde{f}_n(f)(\mbf{x}_i)-\hat{f}_n(\mbf{x}_i)\right)+\|\hat{f}_n-f_0\|_n^2-\|\pi_{r_n}\tilde{f}_n(f)-f_0\|_n^2.
\end{align*}
Now note that
\begin{align*}
\|\pi_{r_n}\tilde{f}_n(f)-f_0\|_n^2 & =\|\pi_{r_n}\tilde{f}_n(f)-\tilde{f}_n(f)+\tilde{f}_n(f)-f_0\|_n^2\\
	& =\|\pi_{r_n}\tilde{f}_n(f)-\tilde{f}_n(f)+(1-\delta_n)\hat{f}_n+\delta_n(f_0+\iota)-f_0\|_n^2\\
	& =\|\pi_{r_n}\tilde{f}_n(f)-\tilde{f}_n(f)+(1-\delta_n)(\hat{f}_n-f_0)+\delta_n\iota\|_n^2\\
	& =\left\langle\pi_{r_n}\tilde{f}_n(f)-\tilde{f}_n(f)+(1-\delta_n)(\hat{f}_n-f_0)+\delta_n\iota,\right.\\
	& \hspace*{3.3cm}\left.\pi_{r_n}\tilde{f}_n(f)-\tilde{f}_n(f)+(1-\delta_n)(\hat{f}_n-f_0)+\delta_n\iota\right\rangle\\
	& =\|\pi_{r_n}\tilde{f}_n(f)-\tilde{f}_n(f)\|_n^2+(1-\delta_n)^2\|\hat{f}_n-f_0\|_n^2+\delta_n^2\\
	& \hspace*{3.3cm}+2(1-\delta_n)\left\langle\pi_{r_n}\tilde{f}_n(f)-\tilde{f}_n(f),\hat{f}_n-f_0\right\rangle\\
	& \hspace*{3.3cm}+2\delta_n\left\langle\pi_{r_n}\tilde{f}_n(f)-\tilde{f}_n(f),\iota\right\rangle\\
	& \hspace*{3.3cm}+2(1-\delta_n)\delta_n\left\langle\hat{f}_n-f_0,\iota\right\rangle\\
	& \leq(1-\delta_n)^2\|\hat{f}_n-f_0\|_n^2+2(1-\delta_n)\left\langle\hat{f}_n-f_0,\delta_n\iota\right\rangle+\delta_n^2\\
	& \hspace*{3.3cm}+2(1-\delta_n)\|\pi_{r_n}\tilde{f}_n(f)-\tilde{f}_n(f)\|_n\|\hat{f}_n-f_0\|_n\\
	& \hspace*{3.3cm}+2\delta_n\|\pi_{r_n}\tilde{f}_n(f)-\tilde{f}_n(f)\|_n+\|\pi_{r_n}\tilde{f}_n(f)-\tilde{f}_n(f)\|_n^2,
\end{align*}
where the last inequality follows from the Cauchy-Schwarz inequality. Since
\begin{align*}
\frac{2}{n}\sum_{i=1}^n\epsilon_i\left(\pi_{r_n}\tilde{f}_n(f)(\mbf{x}_i)-\hat{f}_n(\mbf{x}_i)\right) & =\frac{2}{n}\sum_{i=1}^n\epsilon_i\left(\pi_{r_n}\tilde{f}_n(f)(\mbf{x}_i)-\tilde{f}_n(f)(\mbf{x}_i)+\tilde{f}_n(f)(\mbf{x}_i)-\hat{f}_n(\mbf{x}_i)\right)\\
	& =\frac{2}{n}\sum_{i=1}^n\epsilon_i\left(\pi_{r_n}\tilde{f}_n(f)(\mbf{x}_i)-\tilde{f}_n(f)(\mbf{x}_i)\right)\\
	& \qquad-\frac{2}{n}\delta_n\sum_{i=1}^n\epsilon_i\left(\hat{f}_n(\mbf{x}_i)-f_0(\mbf{x}_i)\right)-\frac{2}{n}\delta_n\sum_{i=1}^n\epsilon_i,
\end{align*}
by the definition of $\hat{f}_n$, we have
\begin{align*}
-\mcal{O}_p(\delta_n^2) & \leq\inf_{f\in\mcal{F}_{r_n}}\mbb{Q}_n(f)-\mbb{Q}_n(\hat{f}_n)\\
	& \leq\mbb{Q}_n(\pi_{r_n}\tilde{f}_n(f))-\mbb{Q}_n(\hat{f}_n)\\
	& =\|\pi_{r_n}\tilde{f}_n(f)-f_0\|_n^2-\|\hat{f}_n-f_0\|_n^2-\frac{2}{n}\sum_{i=1}^n\epsilon_i\left(\pi_{r_n}\tilde{f}_n(f)(\mbf{x}_i)-\hat{f}_n(\mbf{x}_i)\right)\\
	& \leq(1-\delta_n)^2\|\hat{f}_n-f_0\|_n^2-\|\hat{f}_n-f_0\|_n^2+2(1-\delta_n)\delta_n\left\langle\hat{f}_n-f_0,\iota\right\rangle\\
	& \hspace{2cm}+2(1-\delta_n)\|\hat{f}_n-f_0\|_n\|\pi_{r_n}\tilde{f}_n(f)-\tilde{f}_n(f)\|_n\\
	& \hspace{2cm}+2\delta_n\|\pi_{r_n}\tilde{f}_n(f)-\tilde{f}_n(f)\|_n+\|\pi_{r_n}\tilde{f}_n(f)-\tilde{f}_n(f)\|_n^2\\
	& \hspace{2cm}-\frac{2}{n}\sum_{i=1}^n\epsilon_i\left(\pi_{r_n}\tilde{f}_n(f)(\mbf{x}_i)-\tilde{f}_n(f)(\mbf{x}_i)\right)+\frac{2}{n}\delta_n\sum_{i=1}^n\epsilon_i\left(\hat{f}_n(\mbf{x}_i)-f_0(\mbf{x}_i)\right)\\
	& \hspace{2cm}+\frac{2}{n}\delta_n\sum_{i=1}^n\epsilon_i+\mcal{O}_p(\delta_n^2)\\
	& =(-2\delta_n+\delta_n^2)\|\hat{f}_n-f_0\|_n^2\\
	& \hspace{2cm}+2(1-\delta_n)\delta_n\left\langle\hat{f}_n-f_0,\iota\right\rangle+2(1-\delta_n)\|\hat{f}_n-f_0\|_n\|\pi_{r_n}\tilde{f}_n(f)-\tilde{f}_n(f)\|_n\\
	& \hspace{2cm}+2\delta_n\|\pi_{r_n}\tilde{f}_n(f)-\tilde{f}_n(f)\|_n+\|\pi_{r_n}\tilde{f}_n(f)-\tilde{f}_n(f)\|_n^2\\
	& \hspace{2cm}-\frac{2}{n}\sum_{i=1}^n\epsilon_i\left(\pi_{r_n}\tilde{f}_n(f)(\mbf{x}_i)-\tilde{f}_n(f)(\mbf{x}_i)\right)+\frac{2}{n}\delta_n\sum_{i=1}^n\epsilon_i\left(\hat{f}_n(\mbf{x}_i)-f_0(\mbf{x}_i)\right)\\
	& \hspace{2cm}+\frac{2}{n}\delta_n\sum_{i=1}^n\epsilon_i+\mcal{O}_p(\delta_n^2)\\
	& \leq\delta_n^2\|\hat{f}_n-f_0\|_n^2+2(1-\delta_n)\delta_n\left\langle\hat{f}_n-f_0,\iota\right\rangle+2(1-\delta_n)\|\hat{f}_n-f_0\|_n\|\pi_{r_n}\tilde{f}_n(f)-\tilde{f}_n(f)\|_n
\end{align*}
\begin{align*}
	& \hspace{4cm}+2\delta_n\|\pi_{r_n}\tilde{f}_n(f)-\tilde{f}_n(f)\|_n+\|\pi_{r_n}\tilde{f}_n(f)-\tilde{f}_n(f)\|_n^2\\
	& \hspace{4cm}-\frac{2}{n}\sum_{i=1}^n\epsilon_i\left(\pi_{r_n}\tilde{f}_n(f)(\mbf{x}_i)-\tilde{f}_n(f)(\mbf{x}_i)\right)+\frac{2}{n}\delta_n\sum_{i=1}^n\epsilon_i\left(\hat{f}_n(\mbf{x}_i)-f_0(\mbf{x}_i)\right)\\
	& \hspace{4cm}+\frac{2}{n}\delta_n\sum_{i=1}^n\epsilon_i+\mcal{O}_p(\delta_n^2),\numberthis\label{Eq: DeviationOfInnerProd&Sum}
\end{align*}
where the last inequality follows by noting that $(1-\delta_n)^2-1=-2\delta_n+\delta_n^2\leq\delta_n^2$. From the condition (C1), we can get
\begin{align*}
[r_n(d+2)+1] & V_n^2\log[r_nV_n(d+2)+1]\\
	& \leq\left([r_n(d+2)+1]V_n\log[r_nV_n(d+2)+1]\right)^4=o(n).
\end{align*}
Combining with Theorem \ref{Thm: NNSieveConsistency}, we obtain that $\|\hat{f}_n-f_0\|_n=o_p(1)$ and hence $\delta_n^2\|\hat{f}_n-f_0\|_n^2=o_p(\delta_n^2)$. From condition (C3), we have
\begin{align*}
2(1-\delta_n)\|\hat{f}_n-f_0\|_n\|\pi_{r_n}\tilde{f}_n(f)-\tilde{f}_n(f)\|_n & \leq2\|\hat{f}_n-f_0\|_n\|\pi_{r_n}\tilde{f}_n(f)-\tilde{f}_n(f)\|_n\\
	& =\mcal{O}_p\left(\rho_n^{-1}\rho_n\delta_n^2\right)=\mcal{O}_p(\delta_n^2).
\end{align*}
Similarly, since $\rho_n\delta_n=o(1)$, we have
\begin{align*}
2\delta_n\|\pi_{r_n}\tilde{f}_n(f)-\tilde{f}_n(f)\|_n & =\mcal{O}_p(\delta_n\cdot\rho_n\delta_n^2)=o_p(\delta_n^2)\\
\|\pi_{r_n}\tilde{f}_n(f)-\tilde{f}_n(f)\|_n^2 & =\mcal{O}_p(\rho_n^2\delta_n^4)=o_p(\delta_n^2).
\end{align*}
Based on condition (C4), we know that
$$
\frac{2}{n}\sum_{i=1}^n\epsilon_i\left(\pi_{r_n}\tilde{f}_n(f)-\tilde{f}_n(f)\right)=\mcal{O}_p(\delta_n^2),
$$
and from Corollary \ref{Cor: LargeDevEmpProc}, we also have
$$
\frac{2}{n}\delta_n\sum_{i=1}^n\epsilon_i\left(\hat{f}_n(\mbf{x}_i)-f_0(\mbf{x}_i)\right)=o_p(\delta_n\cdot n^{-1/2}).
$$
It follows from these observations that
$$
-2(1-\delta_n)\left\langle\hat{f}_n-f_0,\delta_n\iota\right\rangle+\frac{2\delta_n}{n}\sum_{i=1}^n\epsilon_i\leq\mcal{O}_p(\delta_n^2)+\o_p(\delta_n^2)+o_p(\delta_n\cdot n^{-1/2}),
$$
which implies that
$$
-(1-\delta_n)\left\langle\hat{f}_n-f_0,\iota\right\rangle+\frac{1}{n}\sum_{i=1}^n\epsilon_i\leq\mcal{O}_p(\delta_n)+o_p(n^{-1/2})=o_p(n^{-1/2}).
$$
By replacing $\iota$ with $-\iota$, we can obtain the same result and hence
\begin{align*}
\left|\left\langle\hat{f}_n-f_0,\iota\right\rangle-\frac{1}{n}\sum_{i=1}^n\epsilon_i\right| & \leq\left|(1-\delta_n)\left\langle\hat{f}_n-f_0,\iota\right\rangle-\frac{1}{n}\sum_{i=1}^n\epsilon_i\right|+\delta_n\left|\left\langle\hat{f}_n-f_0,\iota\right\rangle\right|\\
	& \leq o_p(n^{-1/2})+\delta_n\|\hat{f}_n-f_0\|_n\\
	& =o_p(n^{-1/2}).
\end{align*}
Therefore, 
$$
\left\langle\hat{f}_n-f_0,\iota\right\rangle=\frac{1}{n}\sum_{i=1}^n\epsilon_i+o_p(n^{-1/2}),
$$
and the desired result follows from the classical Central Limit Theorem.
\end{proof}

Theorem \ref{Thm: Asymptotic Normality} can be used directly for hypothesis testing of neural network with one hidden layer if we know the variance of the random error $\sigma^2$. In practice, this is rarely the case. To perform hypothesis testing when $\sigma^2$ is unknown, it is natural to find a good estimator of $\sigma^2$ and use a ``plug-in" test statistic. A natural estimator for $\sigma^2$ is
$$
\hat{\sigma}_n^2=\frac{1}{n}\sum_{i=1}^n\left(y_i-\hat{f}_n(\mbf{x}_i)\right)^2=\mbb{Q}_n\left(\hat{f}_n\right).
$$
We then need to establish the asymptotic normality for the statistic $\frac{1}{\hat{\sigma}_n\sqrt{n}}\sum_{i=1}^n\left[\hat{f}_n(\mbf{x}_i)-f_0(\mbf{x}_i)\right]$.

\begin{thm}[Asymptotic Normality for Plug-in Statistic]\label{Thm: Asymptotic Normality Plug-in}
Suppose that $f_0\in C(\mcal{X})$, where $\mcal{X}\subset\mbb{R}^d$ is a compact set and $0\leq\eta_n=o\left(n^{-1}\right)$. Then under the conditions as stated in Theorem \ref{Thm: Asymptotic Normality}, we have
$$
\frac{1}{\hat{\sigma}_n\sqrt{n}}\sum_{i=1}^n\left[\hat{f}_n(\mbf{x}_i)-f_0(\mbf{x}_i)\right]\xrightarrow{d}\mcal{N}(0,1).
$$
\end{thm}

\begin{proof}
Note that
\begin{align*}
\hat{\sigma}_n^2=\mbb{Q}_n(\hat{f}_n) & =\frac{1}{n}\sum_{i=1}^n\left(y_i-\hat{f}_n(\mbf{x}_i)\right)^2=\frac{1}{n}\sum_{i=1}^n\left(f_0(\mbf{x}_i)+\epsilon_i-\hat{f}_n(\mbf{x}_i)\right)^2\\
	& =\frac{1}{n}\sum_{i=1}^n\left(\hat{f}_n(\mbf{x}_i)-f_0(\mbf{x}_i)\right)^2-\frac{2}{n}\sum_{i=1}^n\epsilon_i\left(\hat{f}_n(\mbf{x}_i)-f_0(\mbf{x}_i)\right)+\frac{1}{n}\sum_{i=1}^n\epsilon_i^2\\
	& =\frac{1}{n}\sum_{i=1}^n\epsilon_i^2-\frac{2}{n}\sum_{i=1}^n\epsilon_i\left(\hat{f}_n(\mbf{x}_i)-f_0(\mbf{x}_i)\right)+\|\hat{f}_n-f_0\|_n^2
\end{align*}
Based on the rate of convergence of $\hat{f}_n$ we obtained in Theorem \ref{Thm: Rate of Convergence} and condition (C1), we know that
$$
\left\|\hat{f}_n-f_0\right\|_n^2=\mcal{O}_p^*\left(\max\left\{\|\pi_{r_n}f_0-f_0\|_n^2,\frac{r_n(d+2)\log n}{n}\right\}\right).
$$
Under (C3), $\|\pi_{r_n}f_0-f_0\|_n^2=o\left(\rho_n^2\delta_n^4\right)=o(n^{-1/2})$ and under (C1), we have
\begin{align*}
\left(\frac{r_n(d+2)\log n}{n}\right) & \leq o\left(\frac{n^{1/4}\log n}{n}\right)\\
	& =o\left(\frac{\log n}{n^{3/4}}\right)=o(n^{-1/2}),
\end{align*}
which implies that $\left\|\hat{f}_n-f_0\right\|_n^2=o_p(n^{-1/2})$. Moreover, by the same arguments as in the proof of Theorem \ref{Thm: Asymptotic Normality}, we can show that
$$
\frac{2}{n}\sum_{i=1}^n\epsilon_i\left(\hat{f}_n(\mbf{x}_i)-f_0(\mbf{x}_i)\right)=o_p(n^{-1/2}).
$$
Therefore, 
\begin{align*}
\mbb{Q}_n(\hat{f}_n) & =\frac{1}{n}\sum_{i=1}^n\epsilon_i^2+o_p(n^{-1/2}).
\end{align*}
Based on the Weak Law of Large Numbers, we know that $\frac{1}{n}\sum_{i=1}^n\epsilon_i^2=\sigma^2+o_p(1)$. Therefore,
$$
\hat{\sigma}_n^2=\mbb{Q}_n(\hat{f}_n)=\sigma^2+o_p(1),
$$
and it follows from the Slutsky's Theorem and Theorem \ref{Thm: Asymptotic Normality}, we obtain
$$
\frac{1}{\hat{\sigma}_n\sqrt{n}}\sum_{i=1}^n\left[\hat{f}_n(\mbf{x}_i)-f_0(\mbf{x}_i)\right]=\frac{\sigma}{\hat{\sigma}_n}\frac{1}{\sigma\sqrt{n}}\sum_{i=1}^n\left[\hat{f}_n(\mbf{x}_i)-f_0(\mbf{x}_i)\right]\xrightarrow{d}\mcal{N}(0,1).
$$
\end{proof}

\section{Simulation Studies}
In this section, simulations were conducted to check the validity of the theoretical results obtained in the previous sections. We first used a simple simulation to show that it is hard for the parameter estimators in a neural network with one hidden layer to reach parametric consistency. Then the consistency of the neural network sieve estimators was examined under various simulation scenarios. Finally, we evaluated the asymptotic normality of the neural network sieve estimators.

\subsection{Parameter Inconsistency}\label{Sec: parameter inconsistency}
As mentioned in the introduction, due to the loss of identifiability of the parameters, the parameter estimators obtained from a neural network model are unlikely to be consistent. In this simulation, we use empirical results to confirm such observations. We simulated the response through the following model:
\begin{equation}\label{Eq: ParameterInconsistencyModel}
y_i=f_0(x_i)+\epsilon_i, \quad i=1,\ldots,n,
\end{equation}
where the total sample size $n=500$, $x_1,\ldots,x_n\sim\mrm{i.i.d.}\mcal{N}(0,1)$, $\epsilon_1,\ldots,\epsilon_n\sim\mrm{i.i.d.}\mcal{N}(0,0.1^2)$. The true model is a single-layer neural network with two hidden units.
\begin{equation}\label{Eq: TrueFunction}
f_0(x_i)=-1+\sigma(2x_i+1)-\sigma(-x_i+1),
\end{equation}
When we conducted the simulation, we fitted the simulated data with a single-layer neural network to fit the data by setting the learning rate as 0.1 and performing 3e4 iterations for the back propagation. The cost after 3e4 iterations is 0.0106. Table \ref{Tab: paraInconsistent} summarizes the estimated values of the parameters in this model.

\begin{center}
\begin{threeparttable}
\begin{small}
\captionof{table}{Comparison of the true parameters and the estimated parameters in a single-layer neural network with 2 hidden units.}\label{Tab: paraInconsistent}
\begin{tabular}{cp{1cm}p{1cm}p{1cm}p{1cm}cp{1cm}p{1cm}p{1cm}p{1cm}}
\toprule
\multirow{2}{*}{Estimated Values} & \multicolumn{4}{c}{Weights} & & \multicolumn{3}{c}{Biases}\\\cline{2-5}\cline{7-9}
 & $\gamma_1$ & $\gamma_2$ & $\alpha_1$ & $\alpha_2$ & &  $\gamma_{0,1}$ & $\gamma_{0,2}$ & $\alpha_0$ \\
\hline
True Value & 2.00 & -1.00 & 1.00 & -1.00 & & 1.00 & 1.00 & -1.00\\
Estimated Value & 0.82 & 1.30 & -0.34 & -0.58 & & -0.03 & -0.03 & -1.04\\
\bottomrule
\end{tabular}
\end{small}
\end{threeparttable}
\end{center}
\vspace{0.5cm}

Based on the results in Table \ref{Tab: paraInconsistent}, it is clear that the estimators for most of the weights and biases (except $\alpha_0$) are far from reaching consistency. On the other hand, if we look at the curve of the true function and the curve of the fitted function as shown in Figure \ref{Fig: trueAndFittedCurve}, we can see that most parts are fitted extremely well except for the tail parts. The approximation error $\|\hat{f}-f_0\|_n$ is almost zero as shown in the Figure. This suggests that we should study the asymptotic properties of the estimated function instead of the estimated parameters in the neural network.

\begin{figure}[htbp]
\centering
\includegraphics[width =\textwidth]{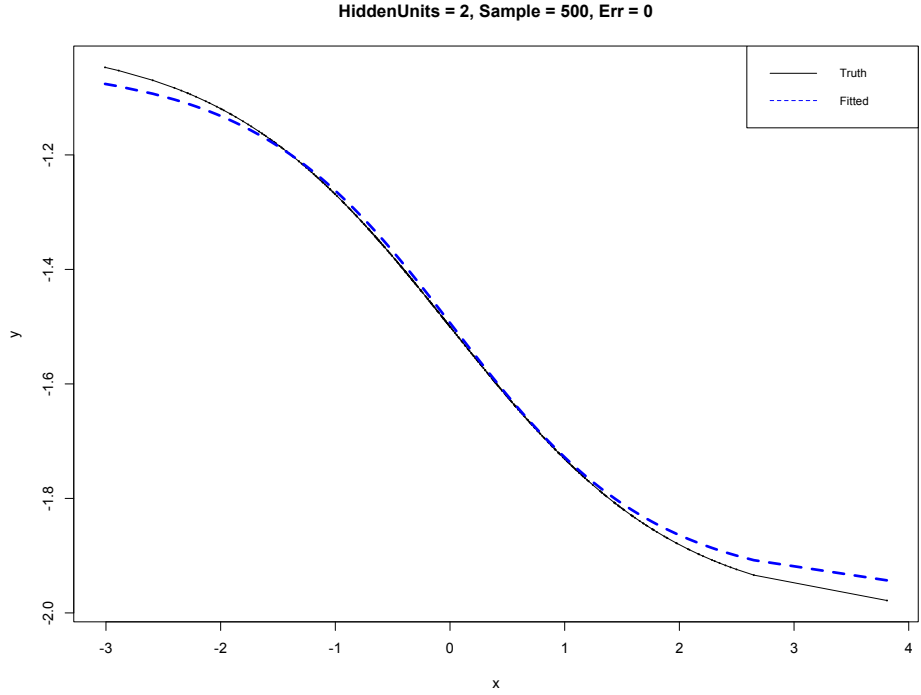}
\caption{Comparison of the true function and the fitted function under the simulation model (\ref{Eq: ParameterInconsistencyModel}). The black curve is the true function defined in (\ref{Eq: TrueFunction}) and the blue dashed curve is the fitted curve obtained after fitting the neural network model. ``Err" stands for the sqaure of empirical $\ell_2$ distance between the esimated function and the true function, that is $Err=(1/n)\sum_{i=1}^n\left[\hat{f}_n(x_i)-f_0(x_i)\right]^2$.}\label{Fig: trueAndFittedCurve}
\end{figure}

\subsection{Consistency for Neural Network Sieve Estimators}\label{Sec: Consistency Simulation}
In this simulation, we are going to check the consistency result from Section \ref{Sec: Consistency} and the validity of the assumption made in Theorem \ref{Thm: NNSieveConsistency}. Based on our construction of the neural network sieve estimators, in each sieve space $\mcal{F}_{r_n}$, there is a constraint on the $\ell_1$ norm for $\mbf{\alpha}$: $\sum_{i=0}^{r_n}|\alpha_i|\leq V_n$. So finding the nearly optimal function in $\mcal{F}_{r_n}$ for $\mbb{Q}_n(f)$ is in fact a constrained optimization problem. A classical way to conduct this optimization is through introducing a Lagrange multiplier for each constraint. Nevertheless, it is usually hard to find an explicit connection between the Lagrange multiplier and the upper bound in the inequality constraint. Instead, we use the subgradient method as discussed in section 7 in \citet{boyd2008note}. The basic idea is to update the parameter $\alpha_0,\ldots,\alpha_{r_n}$ through
$$
\alpha_i^{(k+1)}=\alpha_i^{(k)}-\delta_kg^{(k)},\quad i=0,\ldots,r_n
$$
where $\delta_k>0$ is a step size and $\delta_k$ is chosen to be $0.1/\log(e+k)$ throughout the simulation, which is known as a nonsummable diminishing step size rule. $g^{(k)}$ is a subgradient of the objective or the constraint function $\sum_{j=0}^{r_n}|\alpha_j|-V_n$ at $\mbf{\alpha}^{(k)}$. More specifically, we take
$$
g^{(k)}\in\left\{\begin{array}{ll}
\partial_{\mbf{\alpha}^{(k)}}\mbb{Q}_n(f) & \mrm{if }\sum_{j=0}^{r_n}|\alpha_j|\leq V_n\\
\partial_{\mbf{\alpha}^{(k)}}\sum_{j=0}^{r_n}|\alpha_j| & \mrm{if }\sum_{j=0}^{r_n}|\alpha_j|>V_n
\end{array}
\right..
$$
The updating equations of $\mbf{\gamma}_1,\ldots\mbf{\gamma}_{r_n},\gamma_{0,1},\ldots,\gamma_{0,r_n}$ remains the same as those in the classical gradient descent algorithm. 

We still used equation (\ref{Eq: ParameterInconsistencyModel}) to simulate response, but assumed that the random error $\epsilon_1,\ldots,\epsilon_n$ are i.i.d. $\mcal{N}(0,0.7^2)$. For the true function $f_0(x)$, we considered the following three functions:
\begin{enumerate}[(1)]
\item A neural network with a single hidden layer and 2 hidden units, which is the same as in equation (\ref{Eq: TrueFunction}).
\item A trigonometric function:
	\begin{equation}\label{Eq: trignometric simulation function}
	f_0(x)=\sin\left(\frac{\pi}{3}x\right)+\frac{1}{3}\cos\left(\frac{\pi}{4}x+1\right)
	\end{equation}
\item A continuous function having a non-differential point
	\begin{equation}\label{Eq: non-differential simulation function}
	f_0(x)=\left\{\begin{array}{ll}
	-2x & \mrm{ if }x\leq0\\
	\sqrt{x}\left(x-\frac{1}{4}\right) & \mrm{ if }x>0
	\end{array}
	\right.
	\end{equation}
\end{enumerate}
We then trained a neural network using the subgradient method mentioned above and set the number of iterations used for fitting as 20,000. We chose the growth rate on the number of hidden units $r_n=n^{1/4}$ and the upper bound for $\ell_1$ norm of the weights and bias from the hidden
layer to the output layer $V_n=10n^{1/4}$. Such choice satisfies the condition mentioned in Remark \ref{Rmk: Some Simple Case for Consistency} and hence satisfies the condition in Theorem \ref{Thm: NNSieveConsistency}. We compared the errors $\|\hat{f}_n-f_0\|_n^2$ and the least square errors $\mbb{Q}_n(\hat{f}_n)$ under different sample sizes. The results are summarized in Table \ref{Tab: ConsistencyErr1NN}.

\begin{table}[htbp]
\caption{Comparison of errors $\|\hat{f}_n-f_0\|_n^2$ and the least square errors $\mbb{Q}_n(\hat{f}_n)$ after 20,000 iterations under different sample sizes.}\label{Tab: ConsistencyErr1NN}
\begin{center}
\begin{tabular}{lccccccccc}
\hline
\multirow{2}{*}{Sample Sizes} &  \multicolumn{2}{c}{Neural Network} & & \multicolumn{2}{c}{Sine} & & \multicolumn{2}{c}{Piecewise Continuous}\\
\cline{2-3}\cline{5-6}\cline{8-9}
	&  $\|\hat{f}_n-f_0\|_n^2$ & $\mbb{Q}_n(\hat{f}_n)$ & & $\|\hat{f}_n-f_0\|_n^2$ & $\mbb{Q}_n(\hat{f}_n)$ & & $\|\hat{f}_n-f_0\|_n^2$ & $\mbb{Q}_n(\hat{f}_n)$\\
 \hline
 50 & 3.33E-2 & 0.519 & & 6.04E-2 & 0.513  & & 6.20E-1 & 1.124\\
 100 & 2.79E-2 & 0.552 & & 3.04E-2 & 0.587 & & 3.20E-1 & 0.920\\
 200 & 6.05E-3 & 0.500 & & 1.05E-2 & 0.501 & & 2.51E-1 & 0.786\\
 500 & 8.15E-3 & 0.484 & & 1.19E-2 & 0.499 & & 3.26E-1 & 0.769\\
 1000 & 3.02E-3 & 0.475 & & 1.54E-2 & 0.480 & & 2.98E-2 & 0.489\\
 2000 & 2.88E-3 & 0.500 & & 9.72E-3 & 0.506 & & 1.69E-2 & 0.515\\
 \hline
\end{tabular}
\end{center}
\end{table}

As we can see from Table \ref{Tab: ConsistencyErr1NN}, the errors $\|\hat{f}_n-f_0\|_n^2$ overall has a decreasing pattern as the sample size increases. There are some cases where the error becomes a little bit larger when the sample sizes increases (e.g. the errors using 500 samples in all scenarios is larger than those errors using 200 sample). One explanation is that the number of hidden units increase from 3 (for 200 samples) to 4 (for 500 samples) under our simulation setup, which adds variation to the estimation performance. Overall, the table shows that the estimated function $\hat{f}_n$ is indeed consistent in the sense that $\|\hat{f}_n-f_0\|_n=o_p^*(1)$. Figure \ref{Fig: Consistency1NN} plots the fitted functions and the true function, from which we can straightforwardly visualize the result more and draw the conclusions.

\begin{figure}[htbp]
\centering
\includegraphics[height=0.27\textheight, width=0.8\textwidth]{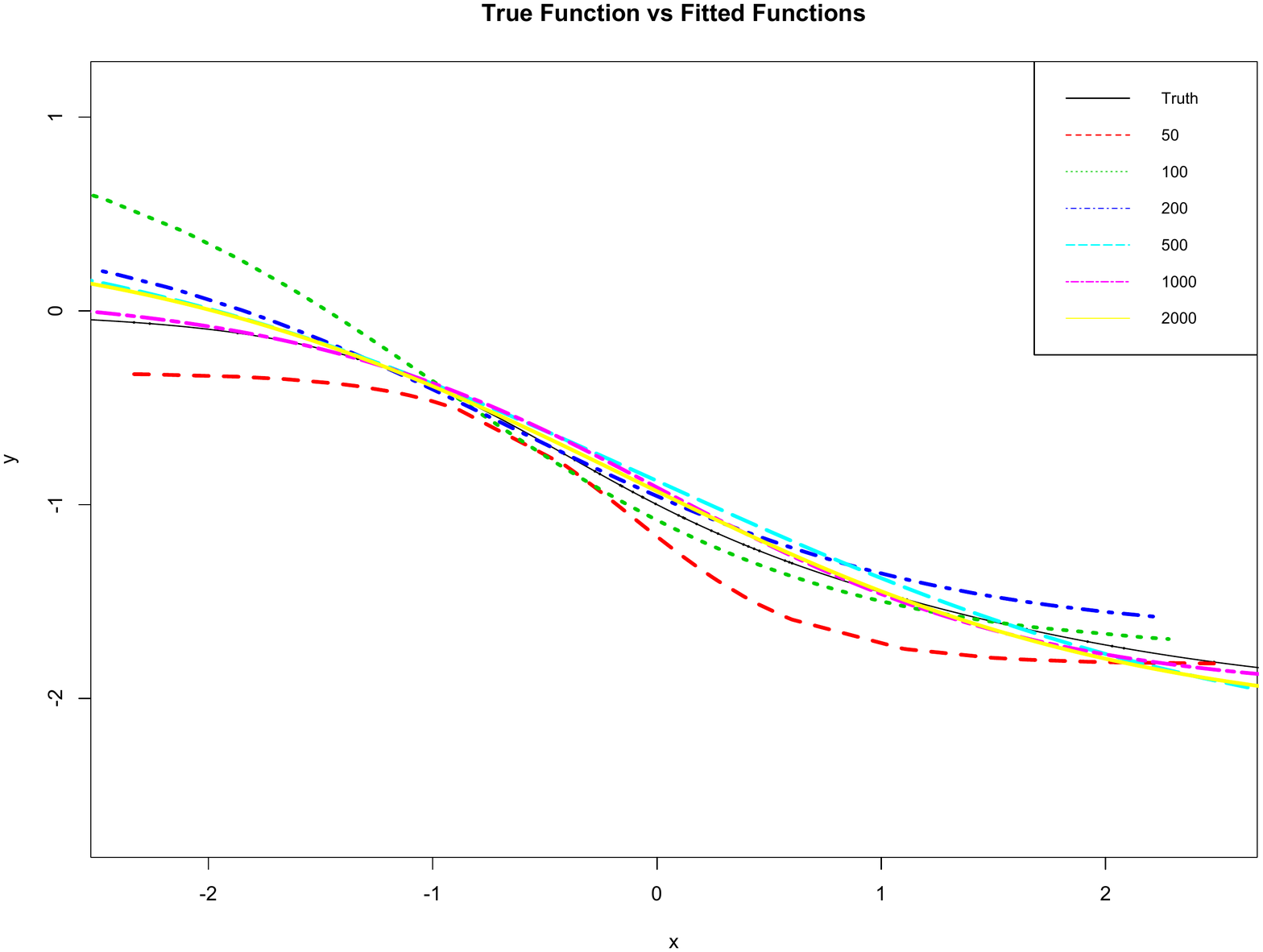}\\
\includegraphics[height=0.27\textheight, width=0.8\textwidth]{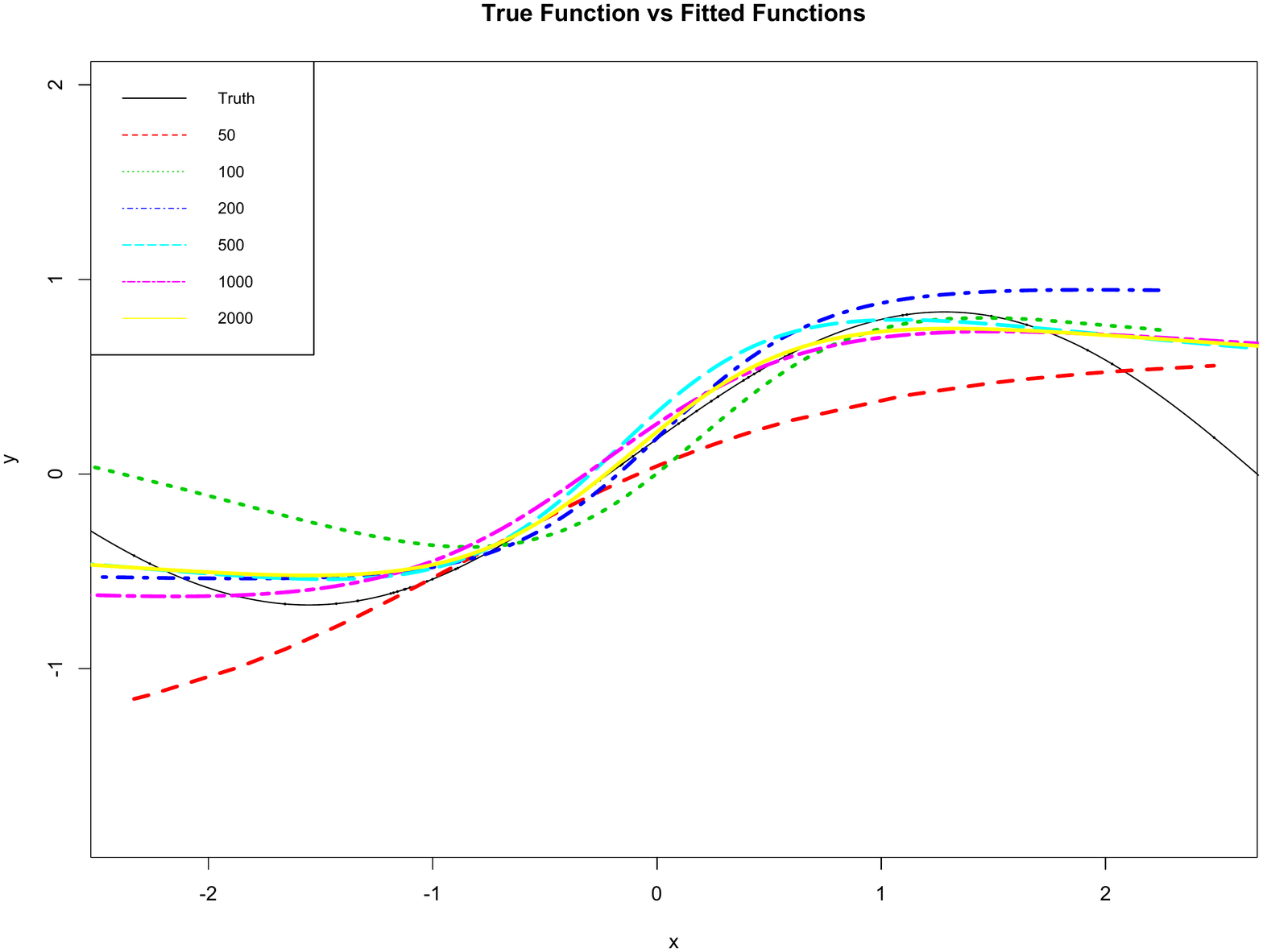}\\
\includegraphics[height=0.27\textheight, width=0.8\textwidth]{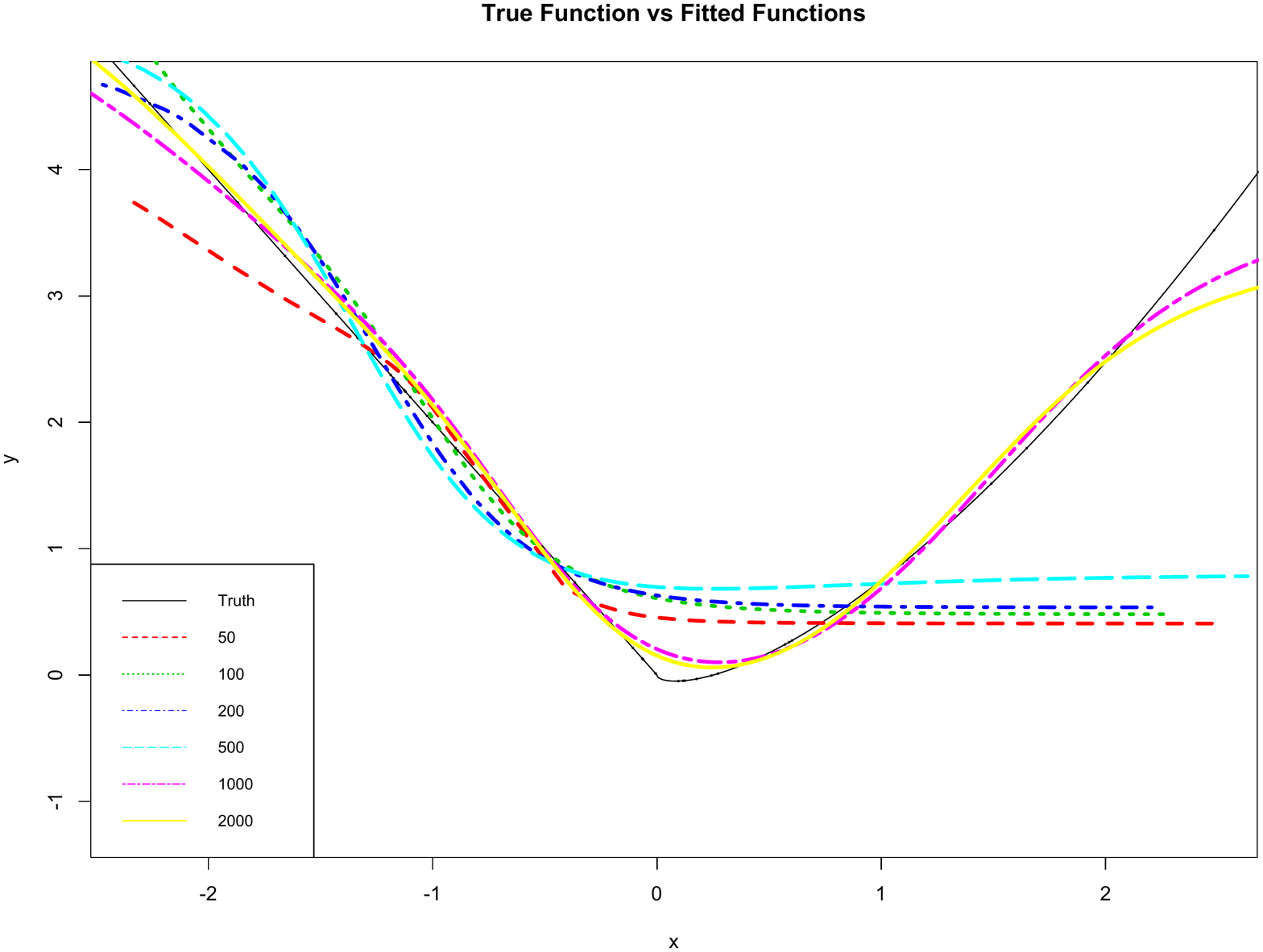}
\caption{Comparison of the true function and the fitted function for three different types of non-linear functions. The top panel shows the scenario when the true function is a single layer neural network; the middle panel shows the scenario when the true function is a sine function, and the bottom panel shows the scenario when the true function is a continuous function having a non-differentiable point. As we can see from all the cases, the fitted curve becomes closer to the truth as the sample size increases.}\label{Fig: Consistency1NN}
\end{figure}

\subsection{Asymptotic Normality for Neural Network Sieve Estimators}
The last part of the simulation focuses on the asymptotic normality derived in Theorem \ref{Thm: Asymptotic Normality}. We still considered the same types of true functions as described in section \ref{Sec: Consistency Simulation} but sampled the random errors from the standard normal distribution. In this simulation, we still used the subgradient method to obtain the fitted model. The number of iterations used for fitting was set at 20,000. What is different from section \ref{Sec: Consistency Simulation} is the growth rates for $r_n$ and $V_n$ set in this simulation. As mentioned in section \ref{Sec: Asymptotic Normality}, the growth rates required for asymptotic normality are slower than those required for consistency. Therefore, we chose $r_n=n^{1/8}$ and $V_n=10n^{1/10}$. Such choice satisfies the condition (C1) in Theorem \ref{Thm: Asymptotic Normality}. In order to get the normal Q-Q plot for $n^{-1/2}\sum_{i=1}^n\left[\hat{f}_n(\mbf{x}_i)-f_0(\mbf{x}_i)\right]$, we repeated the simulation 200 times. 

Figure \ref{QQPlot1} to Figure \ref{QQPlot3} are the normal Q-Q plots under different nonlinear functions and various sample sizes. From the figures, we found that the statistic $n^{-1/2}\sum_{i=1}^n\left[\hat{f}_n(\mbf{x}_i)-f_0(\mbf{x}_i)\right]$ fit the normal distribution pretty well under all simulation scenarios. It is also worth to note that the Q-Q plots looks similar under all simulation scenarios. This is what we would expect since the limiting distribution for the statistic $n^{-1/2}\sum_{i=1}^n\left[\hat{f}_n(\mbf{x}_i)-f_0(\mbf{x}_i)\right]$ is $\mcal{N}(0,1)$ under all scenarios. Another implication we can obtain from the Q-Q plots is that the statistic $n^{-1/2}\sum_{i=1}^n\left[\hat{f}_n(\mbf{x}_i)-f_0(\mbf{x}_i)\right]$ is robust to the choice of $f_0$. Therefore, as long as the true function $f_0$ is continuous, $\mcal{N}(0,1)$ is a good asymptotic distribution for $n^{-1/2}\sum_{i=1}^n\left[\hat{f}_n(\mbf{x}_i)-f_0(\mbf{x}_i)\right]$, which facilitates hypothesis testing.

\begin{figure}[htbp]
\centering
\includegraphics[width=\textwidth]{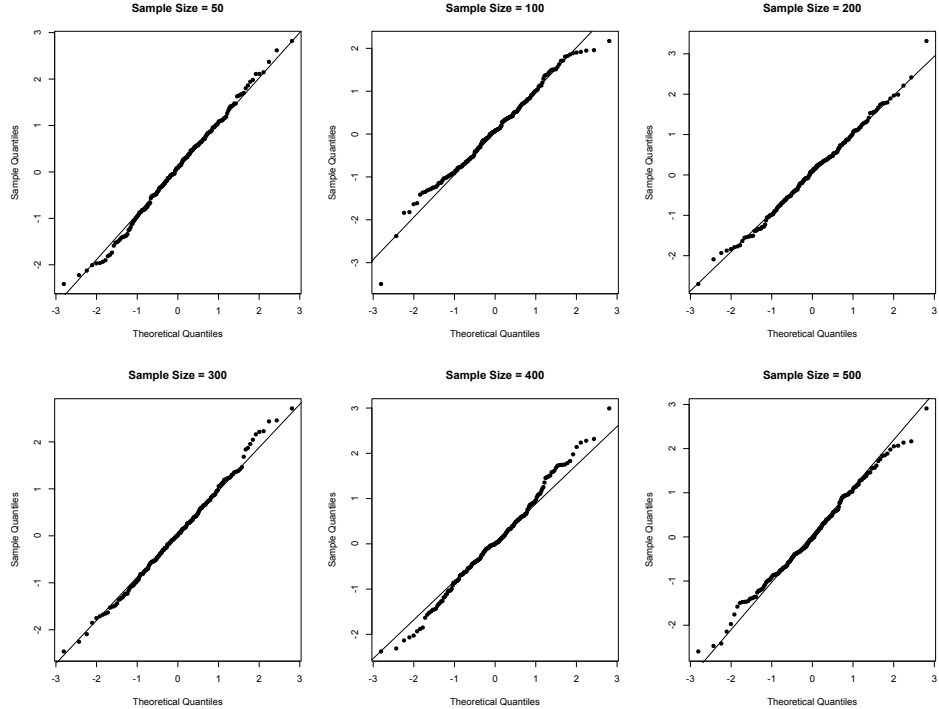}
\caption{Normal Q-Q plot for $n^{-1/2}\sum_{i=1}^n\left[\hat{f}_n(\mbf{x}_i)-f_0(\mbf{x}_i)\right]$ various sample sizes. The true function $f_0$ is a single-layer neural network with 2 hidden units as defined in (\ref{Eq: TrueFunction}).}\label{QQPlot1}
\end{figure}

\begin{figure}[htbp]
\centering
\includegraphics[width=\textwidth]{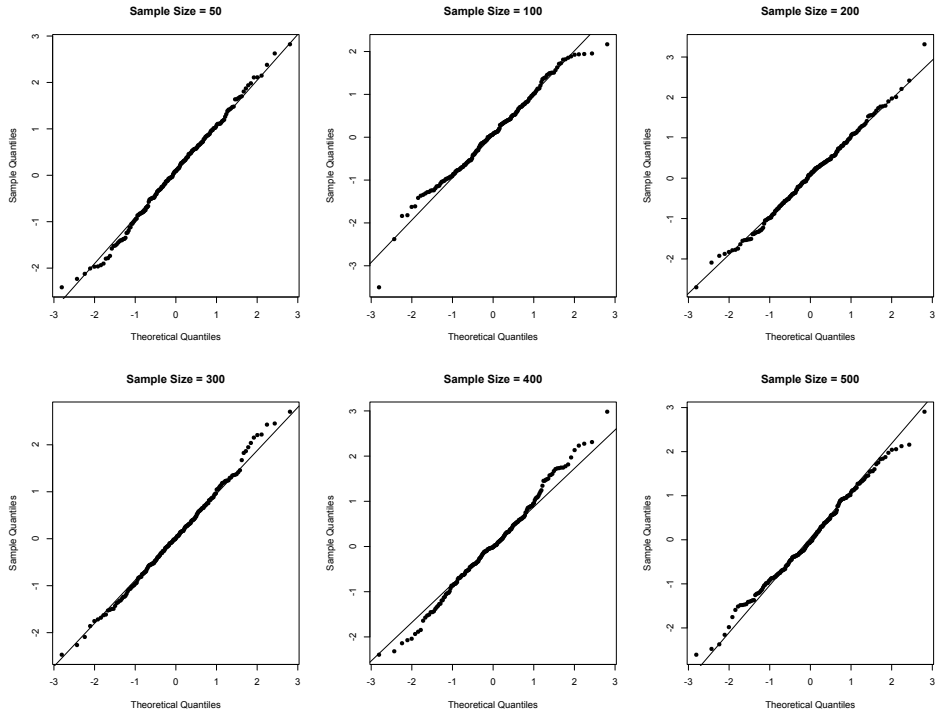}
\caption{Normal Q-Q plot for $n^{-1/2}\sum_{i=1}^n\left[\hat{f}_n(\mbf{x}_i)-f_0(\mbf{x}_i)\right]$ various sample sizes. The true function $f_0$ is a trigonometric function as defined in (\ref{Eq: trignometric simulation function}).}\label{QQPlot2}
\end{figure}

\begin{figure}[htbp]
\centering
\includegraphics[width=\textwidth]{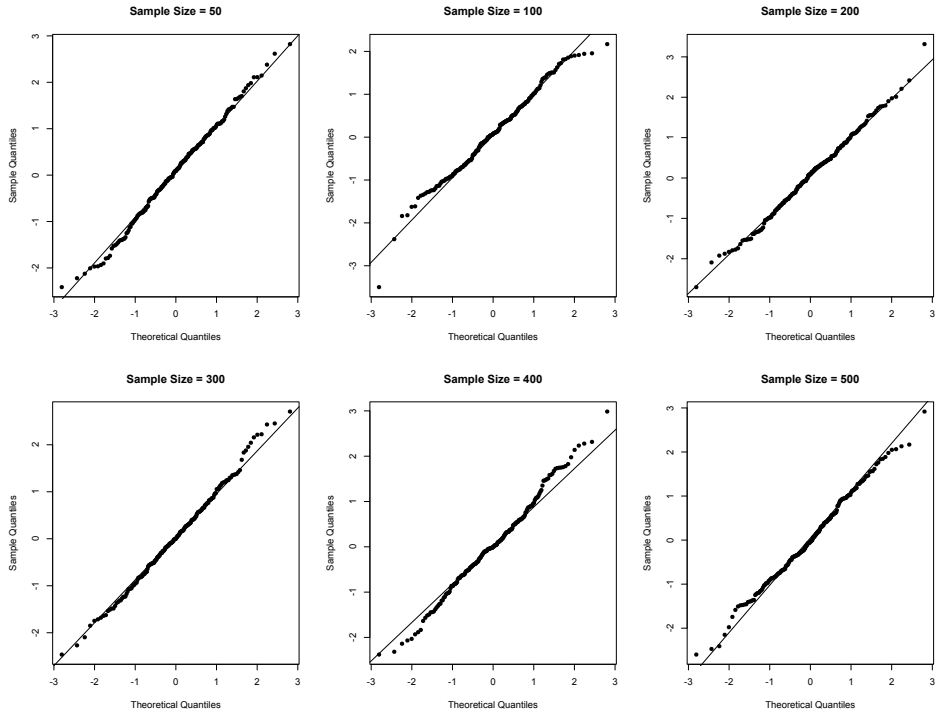}
\caption{Normal Q-Q plot for $n^{-1/2}\sum_{i=1}^n\left[\hat{f}_n(\mbf{x}_i)-f_0(\mbf{x}_i)\right]$ various sample sizes. The true function $f_0$ is a continuous function having a non-differential point as defined in (\ref{Eq: non-differential simulation function}).}\label{QQPlot3}
\end{figure}

Besides the Q-Q plots, we also conducted the normality tests to check whether $n^{-1/2}$ $\sum_{i=1}^n\left[\hat{f}_n(\mbf{x}_i)-f_0(\mbf{x}_i)\right]$ follows the standard normal distribution. Specifically, we used the Shapiro-Wilks test and the Kolmogorov-Smirnov test to perform the normality test. Table \ref{Tab: normal test} summarizes the $p$-values for both normality tests. As we observed from Table \ref{Tab: normal test}, in all cases, we failed to reject that $n^{-1/2}\sum_{i=1}^n\left[\hat{f}_n(\mbf{x}_i)-f_0(\mbf{x}_i)\right]$ follows the standard normal distribution.

\begin{table}[htbp]
\caption{Summary of results from the Shapiro-Wilks test and the Kolmogorov-Smirnov test. We use ``NN", ``TRI" and ``ND" to denote a neural network described in (\ref{Eq: TrueFunction}), a trigonometric function described in (\ref{Eq: trignometric simulation function}) and a continuous function having a non-differential point described in (\ref{Eq: non-differential simulation function}), respectively.}\label{Tab: normal test}
\begin{center}
\begin{tabular}{ccccccccc}
\hline
\multirow{2}{*}{Sample Sizes} &  \multicolumn{3}{c}{Shapiro-Wilks Test} & & \multicolumn{3}{c}{Kolmogorov-Smirnov Test}\\
\cline{2-4}\cline{6-8}
 & NN & TRI & ND & & NN & TRI & ND\\
\hline
50 & 0.878 & 0.884 & 0.881 & & 0.584 & 0.597 & 0.595\\
100 & 0.098 & 0.095 & 0.095 & & 0.472 & 0.508 & 0.484\\
200 & 0.940 & 0.944 & 0.944 & & 0.731 & 0.719 & 0.708\\
300 & 0.884 & 0.888 & 0.872 & & 0.976 & 0.986 & 0.973\\
400 & 0.514 & 0.525 & 0.513 & & 0.670 & 0.754 & 0.708\\
500 & 0.768 & 0.778 & 0.768 & & 0.733 & 0.769 & 0.733\\
\hline
\end{tabular}
\end{center}
\end{table}

\section{Discussion}
We have investigated the asymptotic properties, including the consistency, rate of convergence, and asymptotic normality for neural network sieve estimators with one hidden layer. While in practice, the number of hidden unites is often chosen ad hoc, it is important to note that the conditions in the theorems provide theoretical guidelines on choosing the number of hidden units for a neural network with one hidden layer to achieve the desired statistical properties. The validity of the conditions made in the theorems has also been checked through simulation results. Theorem \ref{Thm: Asymptotic Normality} and Theorem \ref{Thm: Asymptotic Normality Plug-in} can be served as preliminary work for conducting hypothesis testing on $H_0: f_0=h_0$ for a fixed function $h_0$. However, there is currently no simple way to check conditions (C3) and (C4) in the theorem, which requires further researches on local entropy numbers for classes of neural networks.

The work conducted in this paper mainly focuses on sieve estimators based on neural networks with one hidden layer and standard sigmoid activation function. The work presented in this paper can be extended in several ways. The main theorems in this paper depend heavily on the covering number or the entropy number of the function class consisting of neural network with one hidden layer. Theorem 14.5 in \citet{anthony2009neural} provides a general upper bound for the covering number of a function class consisting of deep neural networks with Lipshitz continuous activation functions. Therefore, it is possible to extend our results discussed in this paper to a deep neural network with Lipshitz continuous activation functions. It is also worthwhile to investigate asymptotic properties of other commonly used deep learning models such as convolutional neural networks (CNNs) and recurrent neural networks (RNNs).

When we train a deep neural network, we usually need to face an overfitting issue. In practice, regularization is frequently used to reduce overfitting. Another natural extension of the work discussed in this paper is to modify the loss function by involving some regularization terms. By taking regularization into account, we believe the theories could have a much broader application in real world scenarios.

\section*{Acknowledgment}
X. Shen, C. Jiang, and Q. Lu are supported by the National Institute on Drug Abuse (Award No. K01DA033346
and No. R01DA043501). L. Sakhanenko is supported by NSF grant DMS- 1742881.

\section*{Appendix}
In this appendix, we are going to explore some basic properties of the parameter space $(\mcal{F},\|\cdot\|_n)$ discussed in the main text.

\begin{prop}\label{Prop: Pseudo-metric Parameter space}
The space $(\mcal{F},\|\cdot\|_n)$ is a pseudo-normed space.
\end{prop}

\begin{proof}
Note that $\|f\|_n=\left(\frac{1}{n}\sum_{i=1}^nf^2(\mbf{x}_i)\right)^{1/2}$.
\begin{enumerate}[(i)]
\item Based on the definition of $\|\cdot\|_n$, it is clear that $\|f\|_n\geq0$, for any $f\in\mcal{F}$.
\item For any $\lambda\in\mbb{R}$ and $f\in\mcal{F}$,
	\begin{align*}
	\|\lambda f\|_n & =\left(\frac{1}{n}\sum_{i=1}^n\lambda^2f^2(\mbf{x}_i)\right)^{1/2}=|\lambda|\|f\|_n.
	\end{align*}
\item For any $f,g\in\mcal{F}$,
	\begin{align*}
	\|f+g\|_n & =\left(\frac{1}{n}\sum_{i=1}^n(f(\mbf{x}_i)+g(\mbf{x}_i))^2\right)^{1/2}=\left(\sum_{i=1}^n\left(\frac{1}{\sqrt{n}}f(\mbf{x}_i)+\frac{1}{\sqrt{n}}g(\mbf{x}_i)\right)^2\right)^{1/2}\\
		& \leq\left(\sum_{i=1}^n\left(\frac{1}{\sqrt{n}}f(\mbf{x}_i)\right)^2\right)^{1/2}+\left(\sum_{i=1}^n\left(\frac{1}{\sqrt{n}}g(\mbf{x}_i)\right)^2\right)^{1/2}\\
		& =\|f\|_n+\|g\|_n,
	\end{align*}
	where we applied the triangle inequality to the classical Euclidean norm.
\end{enumerate}
Therefore, we showed that $(\mcal{F},\|\cdot\|_n)$ is a pseudo-normed space.
\end{proof}

\begin{prop}\label{Prop: Inner Product on F}
There is an pseudo-inner product on $\mcal{F}$ such that $\|f\|^2=\langle f,f\rangle$ for any $f\in\mcal{F}$. Moreover, the pseudo-inner product is given by
$$
\langle f,g\rangle=\frac{1}{n}\sum_{i=1}^nf(\mbf{x}_i)g(\mbf{x}_i),\quad\forall f,g\in\mcal{F}.
$$
\end{prop}

\begin{proof}
Based on the theorem attributed to Fréchet, von Neumann and Jordan (see for example, Proposition 14.1.2 in \citet{blanchard2015mathematical}), to show the existence of the inner product, it suffices to check the parallelogram law of the pseudo-norm and the corresponding pseudo-inner product can be obtained via the polarization identity. To check to validity of the parallelogram law, we note that for any $f,g\in\mcal{F}$,
\begin{align*}
\|f+g\|_n^2+\|f-g\|_n^2 & =\frac{1}{n}\sum_{i=1}^n\left(f(\mbf{x}_i)+g(\mbf{x}_i)\right)^2+\frac{1}{n}\sum_{i=1}^n\left(f(\mbf{x}_i)-g(\mbf{x}_i)\right)^2\\
	& =\frac{1}{n}\sum_{i=1}^nf^2(\mbf{x}_i)+\frac{2}{n}\sum_{i=1}^nf(\mbf{x}_i)g(\mbf{x}_i)+\frac{1}{n}\sum_{i=1}^ng^2(\mbf{x}_i)\\
	& \hspace{2cm}+\frac{1}{n}\sum_{i=1}^nf^2(\mbf{x}_i)-\frac{2}{n}\sum_{i=1}^nf(\mbf{x}_i)g(\mbf{x}_i)+\frac{1}{n}\sum_{i=1}^ng^2(\mbf{x}_i)\\
	& =\frac{2}{n}\sum_{i=1}^nf^2(\mbf{x}_i)+\frac{2}{n}\sum_{i=1}^ng^2(\mbf{x}_i)\\
	& =2\|f\|_n^2+2\|g\|_n^2.
\end{align*}
Hence, the parallelogram law is satisfied based on the pseudo-norm, and the pseudo-inner product does exist. By using the polarization identity, for any $f,g\in\mcal{F}$, we have
\begin{align*}
\langle f, g\rangle & =\frac{1}{4}\left(\|f+g\|_n^2-\|f-g\|_n^2\right)\\
	& =\frac{1}{4}\left(\frac{1}{n}\sum_{i=1}^nf^2(\mbf{x}_i)+\frac{2}{n}\sum_{i=1}^nf(\mbf{x}_i)g(\mbf{x}_i)+\frac{1}{n}\sum_{i=1}^ng^2(\mbf{x}_i)-\frac{1}{n}\sum_{i=1}^nf^2(\mbf{x}_i)\right.\\
	& \hspace{6cm}\left.+\frac{2}{n}\sum_{i=1}^nf(\mbf{x}_i)g(\mbf{x}_i)-\frac{1}{n}\sum_{i=1}^ng^2(\mbf{x}_i)\right)\\
	& =\frac{1}{n}\sum_{i=1}^nf(\mbf{x}_i)g(\mbf{x}_i).
\end{align*}
\end{proof}

Let
$$
\mcal{G}=\left\{g:\mbb{R}\to\mbb{R}, \int\left|g'(z)\right|\mrm{d}z\leq M\right\}
$$
be the class of functions of bounded variation in $\mbb{R}$ (see Example 9.3.3 in \citet{van2000empirical}). The following proposition shows that $\mcal{F}_{r_n}\subset\mcal{G}$ for a fixed $n$.

\begin{prop}\label{Prop: BoundedVariation}
For a fixed $n$, $\mcal{F}_{r_n}\subset\mcal{G}$.
\end{prop}

\begin{proof}
For any $f\in\mcal{F}_{r_n}$, we have
$$
f(x)=\alpha_0+\sum_{j=1}^{r_n}\alpha_j\sigma\left(\gamma_j x+\gamma_{0,j}\right),
$$
so that
$$
f'(x)=\sum_{j=1}^{r_n}\alpha_j\gamma_j\sigma\left(\gamma_jx+\gamma_{0,j}\right)\left[1-\sigma\left(\gamma_jx+\gamma_{0,j}\right)\right].
$$
Without loss of generality, we assume that $\gamma_j\neq0$ for $j=1,\ldots,r_n$. Note that
\begin{align*}
\int|f'(x)|\mrm{d}x & =\int\left|\sum_{j=1}^{r_n}\alpha_j\gamma_j\sigma\left(\gamma_jx+\gamma_{0,j}\right)\left[1-\sigma\left(\gamma_jx+\gamma_{0,j}\right)\right]\right|\mrm{d}x\\
	& \leq\sum_{j=1}^{r_n}|\alpha_j||\gamma_j|\int\sigma(\gamma_jx+\gamma_{0,j})[1-\sigma(\gamma_jx+\gamma_{0,j})]\mrm{d}x\\
	& \leq\sum_{j=1}^{r_n}|\alpha_j|\frac{|\gamma_j|}{\gamma_j}\int\sigma(u_j)(1-\sigma(u_j))\mrm{d}u_j,
\end{align*}
where in the last inequality, we let $u_j=\gamma_jx+\gamma_{0,j}$. Clearly, $|\gamma_j|/\gamma_j=\mrm{sign}(\gamma_j)$. Moreover, since
\begin{align*}
\int\sigma(x)(1-\sigma(x))\mrm{d}x & =\int\frac{e^x}{(1+e^x)^2}\mrm{d}x\\
	& =\int_0^\infty\frac{1}{(1+u)^2}\mrm{d}u\quad(\mrm{by letting }u=e^x)\\
	& =\left.-(1+u)^{-1}\right|_0^\infty\\
	& =1,
\end{align*}
for a fixed $n$, we have
\begin{align*}
\int |f'(x)|\mrm{d}x & \leq\sum_{j=1}^{r_n}|\alpha_j|\mrm{sign}(\gamma_j)\leq\sum_{j=1}^{r_n}|\alpha_j|\leq V_n.
\end{align*}
Therefore, $f\in\mcal{G}$ and the desired result follows.
\end{proof}

\bibliographystyle{imsart-nameyear}
\bibliography{NNRef}

\end{document}